\documentclass[a4paper,10pt]{article}
\usepackage{a4}
\usepackage[ansinew]{inputenc}
\usepackage{epsfig}
\usepackage{amssymb}
\usepackage{bbm}
\usepackage{amsmath}
\usepackage{amsthm} 
\usepackage{amscd}
\usepackage{enumerate}
\usepackage{float}
\usepackage{xpatch}

  \newcommand{\E}{\mathcal{E}}
  
  \newcommand{\R}{\mathcal{R}}  
  \newcommand{\A}{\mathcal{A}}
  \newcommand{\D}{\mathcal{D}}

\makeatletter
\xpatchcmd{\proof}{\@addpunct{.}}{\normalfont\,\@addpunct{:}}{}{}
\makeatother

  \newtheoremstyle{dotless}{}{}{\itshape}{}{\bfseries}{:}{ }{}

    \newtheoremstyle{dotlessrem}{}{}{}{}{\bfseries}{:}{ }{}

  \theoremstyle{dotless}

\newtheorem{lemma}{Lemma}[section]

\newtheorem{theorem}[lemma]{Theorem}

 \theoremstyle{dotlessrem}

\renewenvironment{abstract}
 {\small
  \begin{center}
  \bfseries \abstractname\vspace{-.5em}\vspace{0pt}
  \end{center}
  \list{}{
    \setlength{\leftmargin}{1.5cm}%
    \setlength{\rightmargin}{\leftmargin}%
  }%
  \item\relax}
 {\endlist}

	\makeatletter
\newcommand{\pushright}[1]{\ifmeasuring@#1\else\omit\hfill$\displaystyle#1$\fi\ignorespaces}

\begin{document}
\begin{center}
\Large\textbf{Oscillations on the Stretched Sierpinski Gasket}\\
\large Elias Hauser\footnote{Institute of Stochastics and Applications, University of Stuttgart, Pfaffenwaldring 57, 70569 Stuttgart, Germany, E-mail: elias.hauser@mathematik.uni-stuttgart.de}\end{center}
\begin{abstract}
In \cite{haus17} we calculated the leading term in the asymptotics of the eigenvalue counting function for operators coming from a completely symmetric resistance form on the Stretched Sierpinski Gasket (SSG). These resistance forms were introduced in \cite{afk17}. In this work we want to refine the results from \cite{haus17}. The next question that arises is if there are oscillations in the leading term which are typical for highly symmetrical fractals. The SSG is not self-similar but it still exhibits very high symmetry. We have to distinguish between the existence of a periodic function in front of the leading term and oscillations in general. The first one is unlikely as we will see, however the second one still holds. This means there are oscillations in the leading term, but these will not have this very strict periodic behaviour that we know of the Sierpinski Gasket. We will show, that there exist localized eigenfunctions on the SSG which have eigenvalues with very high multiplicities. However in contrast to the self-similar case this fact alone is not enough to show that there can't be convergence. We need to use another method to show the existence of oscillations.
\end{abstract}


\section{Introduction}

Spectral asymptotics is an important tool in physics, for example to understand how heat or waves propagate through media. To answer this questions we need a laplacian. On fractals we can construct such operators with the help of resistance forms. These were introduced by Kigami in \cite{kig03}. If we add a locally finite Borel measure with full support we get Dirichlet forms and thus self-adjoint operators. We are interested in the eigenvalues of these operators. In particular in the asymptotic growing of the eigenvalue counting functions. For the Sierpinski Gasket Shima \cite{shim} and Fukushima-Shima \cite{fushim} calculated the eigenvalues via the eigenvalue decimation method. They found out that
\begin{align}
0<\liminf_{x\rightarrow\infty}N_D(x)x^{-\frac 12 d_S} \leq \limsup_{x\rightarrow\infty} N_D(x)x^{-\frac 12 d_S}<\infty \label{eqasym}
\end{align}
with $d_S=\frac{\ln 9}{\ln 5}$. This disproved a conjecture by Berry in \cite{ber1,ber2} where he proposed that the exponent $d_S$ in the asymptotic growing should be the Hausdorff dimension of the set as a direct generalization of the smooth case of laplacians on open bounded sets $\Omega$ in $\mathbb{R}^n$. From Weyl \cite{we11} we know, that in this case $\liminf$ and $\limsup$ coincide and 
\begin{align*}
\lim_{x\rightarrow\infty} N_D^\Omega(x)x^{-\frac n2}=\frac{\tau_n}{(2\pi)^n}\operatorname{Vol}_n(\Omega)
\end{align*}
where $\tau_n$ is the volume of the unit ball in $\mathbb{R}^n$.

$d_S$ in (\ref{eqasym}) was later calculated for a larger set of fractals by Kigami and Lapidus \cite{kl93}, namely the so called p.c.f. self-similar sets. Moreover in very symmetric cases they showed that there is periodic behaviour in front of the leading term. This periodic function, however, could still be constant. For the Sierpinski Gasket Shima and Fukushima-Shima \cite{shim,fushim} showed with spectral decimation that 
\begin{align}
\liminf_{x\rightarrow\infty}N_D(x)x^{-\frac 12 d_S} < \limsup_{x\rightarrow\infty} N_D(x)x^{-\frac 12 d_S} \label{eqasym2}
\end{align}
But this method is only valid for a very restricted set of fractals. Later Barlow and Kigami showed (\ref{eqasym2}) in \cite{bk97} by using arguments that utilize the symmetry of the set. They showed the existence of localized eigenfunctions. These are eigenfunctions that are only supported on a proper subset. This is a very interesting phenomena. Such localized eigenfunctions give us solutions to the heat and wave equation where the heat and energy stay in the subset. This represents for example perfect heat insulated or sound proofed rooms. The associated eigenvalues have very high multiplicities, which lead to high jumps in the eigenvalue counting function. We can use this to show (\ref{eqasym2}). \\

In this work we consider the Stretched Sierpinski Gasket (also Hanoi attractor) which is a non self-similar but still highly symmetric set. In \cite{haus17} we calculated the leading term of the asymptotics. That means we showed (\ref{eqasym}) and calculated $d_S$. We are now interested in the refinements. The goal is to see if periodicity 
still holds and if (\ref{eqasym2}) is true. We will use the method of \cite{bk97} and look for localized eigenfunctions with high multiplicities. \\

This work is structured as follows. In chapter~\ref{chap2} we will define the Stretched Sierpinski Gasket and briefly review the construction of so called completely symmetric resistance forms which were introduced in \cite{afk17}. We then construct measures to get Dirichlet forms and state the results concerning the leading term in the asymptotics of the eigenvalue counting function. In chapter~\ref{chapsirp} we want to look again at the Sierpinski Gasket and review the results on periodicity and oscillations to be able to compare them to the Stretched Sierpinski Gasket. In chapter~\ref{chapmeas} we analyze the measures and in chapter~\ref{chapper} the energy. We follow the ideas of chapter~\ref{chapsirp} and show why renewal theory isn't applicable and periodicity in general is unlikely. Chapter~\ref{chaple} is the main part of this work. We show the existence of localized eigenfunctions and get estimates on the associated eigenvalues. This gives us a sequence of eigenvalues with which we can show (\ref{eqasym2}). We want to emphasize that (\ref{eqasym2}) doesn't need strict periodicity as for the Sierpinski Gasket. We have oscillations but they may be not as regular and periodic as in the self-similar case. We close this work in chapter~\ref{special} where we introduce more symmetry. These are special cases of the resistance forms from \cite{afk17}. With the additional symmetry we can show that in these cases periodicity holds.

\section{Stretched Sierpinski Gasket: Definition, Dirichlet forms and spectral asymptotics}\label{chap2}
\subsection{Definition Stretched Sierpinski Gasket}
The Stretched Sierpinski Gasket is a non self-similar set, that still exhibits a lot of symmetry. In \cite{af12} the set was analyzed geometrically by Alonso-Ruiz and Freiberg by calculating its Hausdorff dimension. Another name is Hanoi attractor. This is due to its connection to the game "The towers of Hanoi", which can also be found in \cite{af12}.

Let $p_1,p_2,p_3$ be the vertex points of an equilateral triangle with side length 1 and for $\alpha \in(0,1)$
\begin{align*}
G_i(x)&:=\frac{1-\alpha}2(x-p_i) +p_i, \ i\in \{1,2,3\}=:\mathcal{A}	\\[0.1cm]
e_1&:=\{\lambda G_2(p_3)+(1-\lambda)G_2(p_3) \ : \lambda \in  (0,1)\}\quad e_2,e_3 \ \text{analogous}
\end{align*}
Then there exists a unique compact set $K_\alpha$ with
\begin{align*}
K_\alpha = G_1(K_\alpha) \cup G_2(K_\alpha) \cup G_3(K_\alpha) \cup e_1\cup e_2\cup e_3
\end{align*}
This set is called \textit{Stretched Sierpinski Gasket (SSG)} since the contraction ratios are smaller than the ones of the Sierpinski Gasket and the gaps are filled with one-dimensional lines (see Figure~\ref{hanoi}).

The sets $K_\alpha$ for $\alpha\in(0,1)$ are pairwise homeomorphic \cite[Prop. 3.4]{afk17} and since the resistance forms only depend on the topology of $K_\alpha$ we can omit the parameter $\alpha$ in the notation.
\begin{figure}[H]
\centering
\includegraphics[scale=0.1]{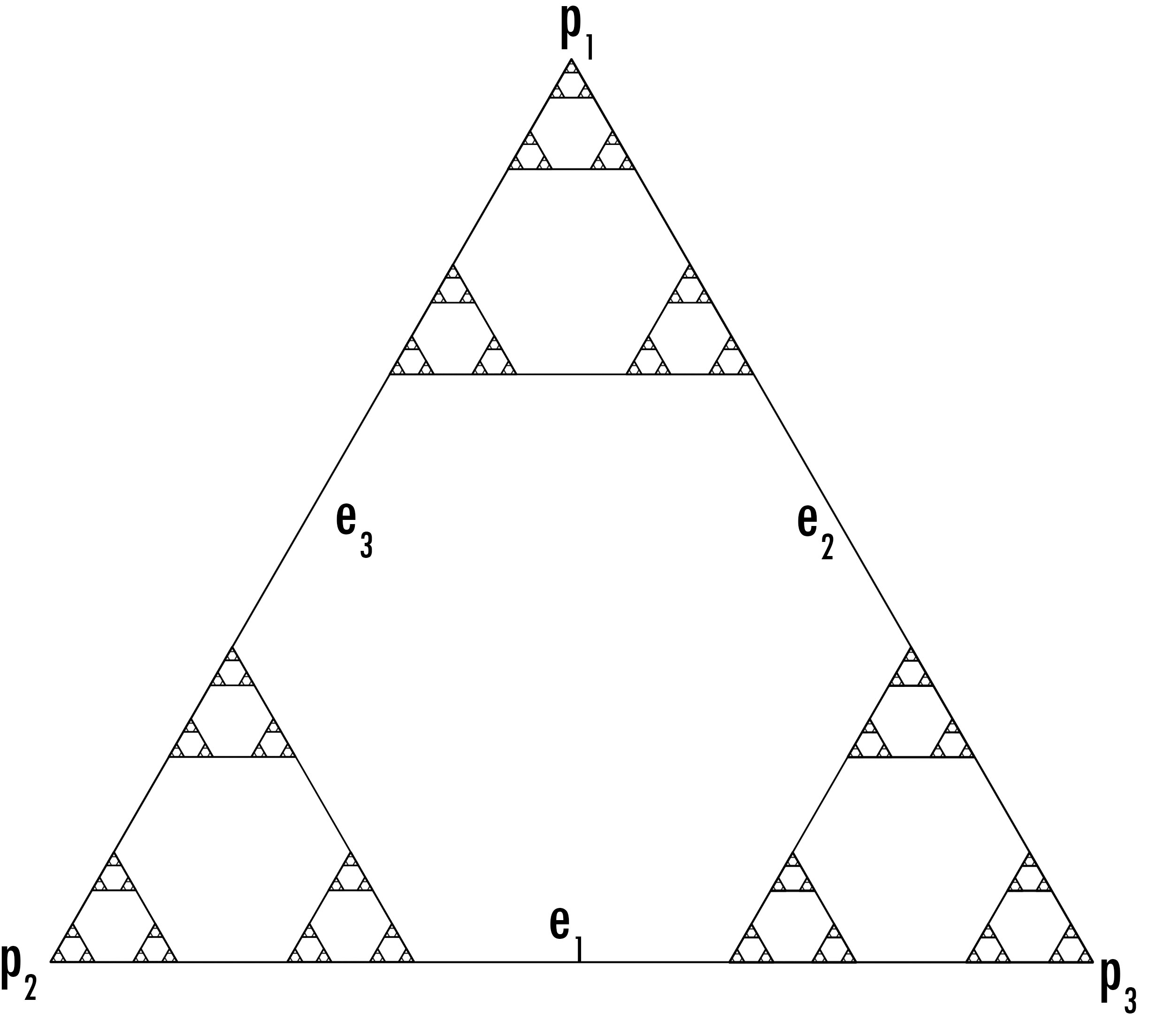}
\caption{The Stretched Sierpinski Gasket.}
\label{hanoi}
\end{figure}
We need to introduce some common notation. Let $\mathcal{A}:=\{1,2,3\}$, for $w\in\mathcal{A}^m$ with $m\in\mathbb{N}_0$:
\begin{itemize} 
\item $G_w=G_{w_1}\circ\ldots \circ G_{w_m}$ (with $G_w=\operatorname{id}$ for the empty word $w\in\mathcal{A}^0$)
\item $V_0:=\{p_1,p_2,p_3\}$,  $V_m:=\bigcup_{w\in \mathcal{A}^m}G_w(V_0)$ 
\item $e_i^w:=G_w(e_i)$
\item $K_w:=G_w(K)$,  $K_m:=\bigcup_{w\in\mathcal{A}^m}K_w$
\item $J_m:=K\backslash K_m$
\end{itemize}
\subsection{Resistance forms on SSG}
To be able to study analysis on the Stretched Sierpinski Gasket we need to introduce a resistance form on $K$. A definition of resistance forms can be found in \cite{kig12}. The choice of the resistance form is not unique and so we get different operators and different spectral asymptotics. The construction of these resistance forms was carried out in \cite{afk17}. The following paragraph will include a brief recapitulation of this construction.
\begin{figure}[H]
\centering
\includegraphics[scale=0.15]{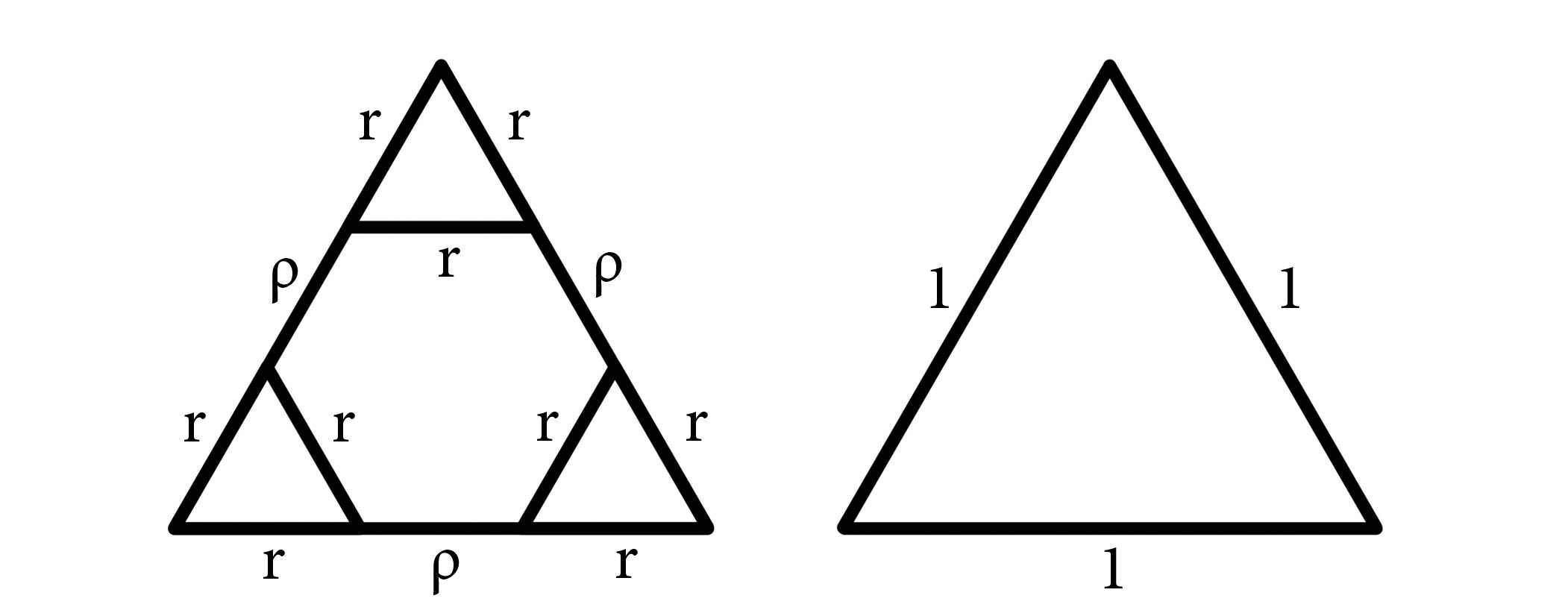}
\caption{Resistances in the first graph approximation.}
\label{resis}
\end{figure}
In Figure~\ref{resis} you can see the first graph approximation of $K$ beside the graph that just contains the connected vertices $p_1,p_2$ and $p_3$. Due to symmetry we want to have the resistances on the smaller triangles all equal $r$ and also all equal $\rho$ on the edges adjoining them. This electrical network should be equivalent to the one on the right with all resistances equal $1$. A quick calculation with the help of the $\Delta-Y$-transformation leads to 
\begin{align*}
\frac 53 r+\rho=1
\end{align*}
Such a pair $(r,\rho)$ is then called a matching pair. In the next graph approximation the smaller triangles get divided further in the same fashion. 
\begin{figure}[H]
\centering
\includegraphics[scale=0.15]{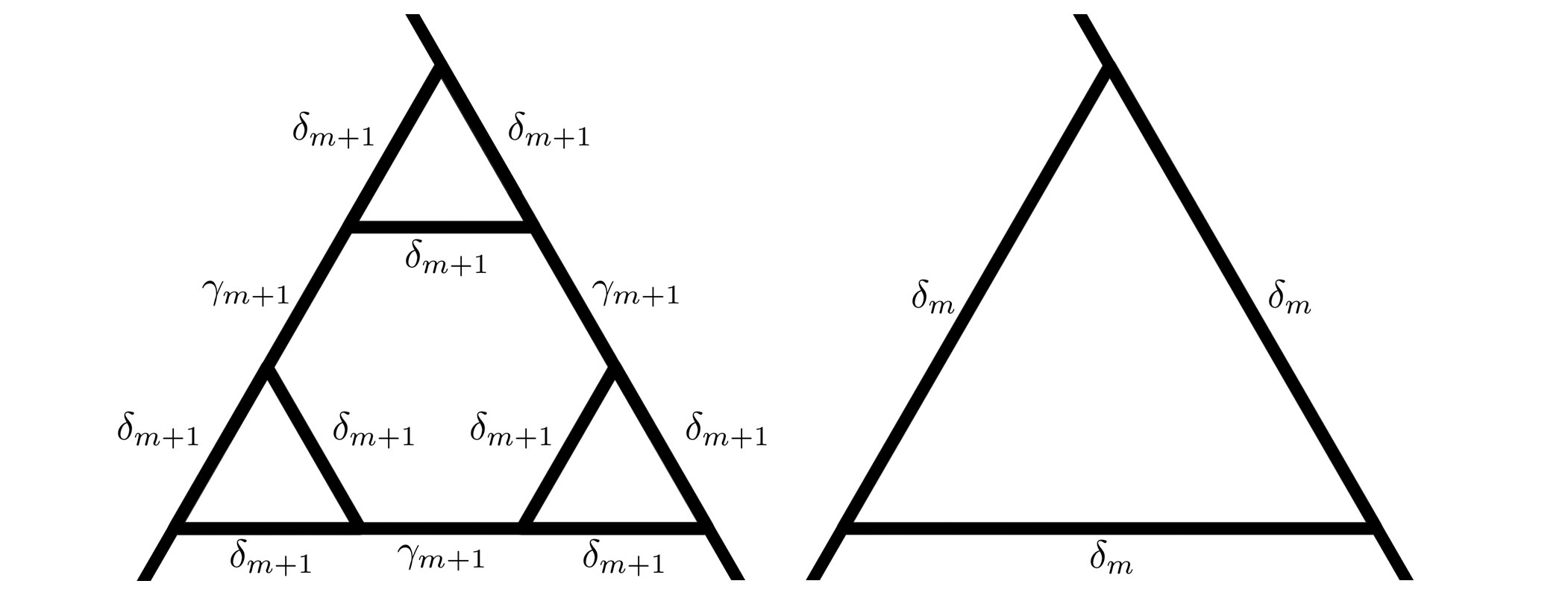}
\caption{Resistances in the $\protect m+1$ graph approximation.}
\label{resis2}
\end{figure}
In general, in the $m+1$ graph approximation, the left triangle in Figure~\ref{resis2} has to be equivalent to the right one with all resistances $\delta_m$. 
The same calculation as for the first graph approximation shows, that it has to hold that 
\begin{align*}
\delta_{m+1}=\delta_m\cdot r_{m+1} \hspace*{0.5cm} \text{and} \hspace*{0.5cm} \gamma_{m+1}=\delta_m\cdot\rho_{m+1}
\end{align*}
with a matching pair $(r_{m+1},\rho_{m+1})$ i.e. $\frac 53 r_{m+1}+ \rho_{m+1}=1$. Notice that the resistances of the edges connecting adjoining cells from the previous graph approximations do not change. We get for the $m$-th graph approximation, that
\begin{align*}
\delta_m=r_1\cdots r_m\hspace*{0.5cm} \text{and} \hspace*{0.5cm} \gamma_m=r_1\cdots r_{m-1} \rho_m
\end{align*}
with $\frac 53 r_i+\rho_i=1$ for all $i$. Such a sequence $\mathcal{R}=(r_i,\rho_i)_{i\geq 1}$ of matching pairs is also called a compatible sequence because each of those sequences will lead to a resistance form on $K$. 

With these resistances we can define a quadratic form. This form will consist of two parts. One part is very similar to the usual resistance form on the Sierpinski Gasket. For $u\in \ell(K)$ define
\begin{align*}
Q_0^\Sigma(u,u)&:=(u(p_1)-u(p_2))^2+(u(p_2)-u(p_3))^2+(u(p_3)-u(p_1))^2\\
Q_m^\Sigma(u,u)&:=\sum_{w\in\mathcal{A}^m} Q_0^\Sigma(u\circ G_w,u\circ G_w)\\[0.2cm]
\mathcal{E}^\Sigma_{\mathcal{R}}(u,u)&:=\lim_{m\rightarrow \infty} \frac 1{\delta_m}Q_m^\Sigma (u,u)
\end{align*}
However this form ignores the adjoining edges of the Stretched Sierpinski Gasket. To get a form on the whole $K$ we need a second part. This can be achieved with the usual one-dimensional Dirichlet energy summed over all edges.

With $\xi_{e^w_i}(t)=(1-t)(e^w_i)_-+t(e^w_i)^+$, $t\in (0,1)$, where $(e^w_i)_-$ and $(e^w_i)^+$ are the endpoints of $e^w_i$, we define
\begin{align*}
\mathcal{D}^I_k(u,u)&:=\sum_{ \begin{array}{c} w\in\mathcal{A}^{k-1}\\ i\in\{1,2,3\}\end{array} } \int_0^1 \left(\frac{d(u\circ \xi_{e^w_i})}{dx}\right)^2 dx\\[0.1cm]
\mathcal{E}^I_{\mathcal{R}}(u,u)&:=\sum_{k=1}^\infty \frac 1{\gamma_k} \mathcal{D}^I_k(u,u)
\end{align*}
Now we define the sum of the two parts as our final quadratic form:
\begin{align*}
\mathcal{E}_{\mathcal{R}}(u,u):=\mathcal{E}^\Sigma_{\mathcal{R}}(u,u)+ \mathcal{E}^I_{\mathcal{R}}(u,u)
\end{align*}
The form $\mathcal{E}_{\mathcal{R}}$ is defined on
\begin{align*}
\mathcal{F}_{\mathcal{R}}=\left\{u\in C(K) : \mathcal{E}_{\mathcal{R}}(u,u)<\infty, \  u|_{e^w_i}\in H^1(e_i^w), \forall i\in\A, w\in\mathcal{A}^m, m\in\mathbb{N}_0\right\}
\end{align*}
where $H^1(e_i^w)=\{u\in \ell(e_i^w), u\circ \xi_{e_i^w}\in H^1(0,1)\}$. One of the results of \cite{afk17} is that for a sequence of matching pairs $\mathcal{R}=(r_i,\rho_i)_{i\geq 1}$ the form $(\mathcal{E}_{\mathcal{R}},\mathcal{F}_{\mathcal{R}})$ is indeed a regular resistance form (see \cite[Theorem 7.16]{afk17}).

The construction of these resistance forms can be studied in much greater detail in \cite{afk17}.
\subsection{Measures and Dirichlet forms}
Until now we have resistance forms. To get Dirichlet forms and thus operators on SSG we need to introduce measures. These measures have to be locally finite Borel-measures with full support. Since we are working on a compact set we are looking for finite measures.\\

We want to describe the measures as the sum of two parts. These parts represent the fractal and the line part in accordance to the geometric appearance of SSG. It is clear how the measure on the fractal part has to look like. We use the normalized Hausdorff measure on $K$ by distributing mass equally to the cells.
\begin{align*}
\mu_f(K_w):=\left(\frac 13\right)^{|w|}
\end{align*}
This is a finite measure but it does not have full support. It is only supported on a proper subset of $K$, namely the attractor of the similitudes $G_1,G_2,G_3$ alone. The measure $\mu_f$ is too rough to measure the one-dimensional lines. Therefore we need a second part on these lines. We want this line part of the measure to fulfil the same symmetries as $K$. We start by assigning some value $a>0$ to each of the edges $e_1,e_2$ and $e_3$.
\begin{figure}[H]
\centering
\includegraphics[width=0.6\textwidth]{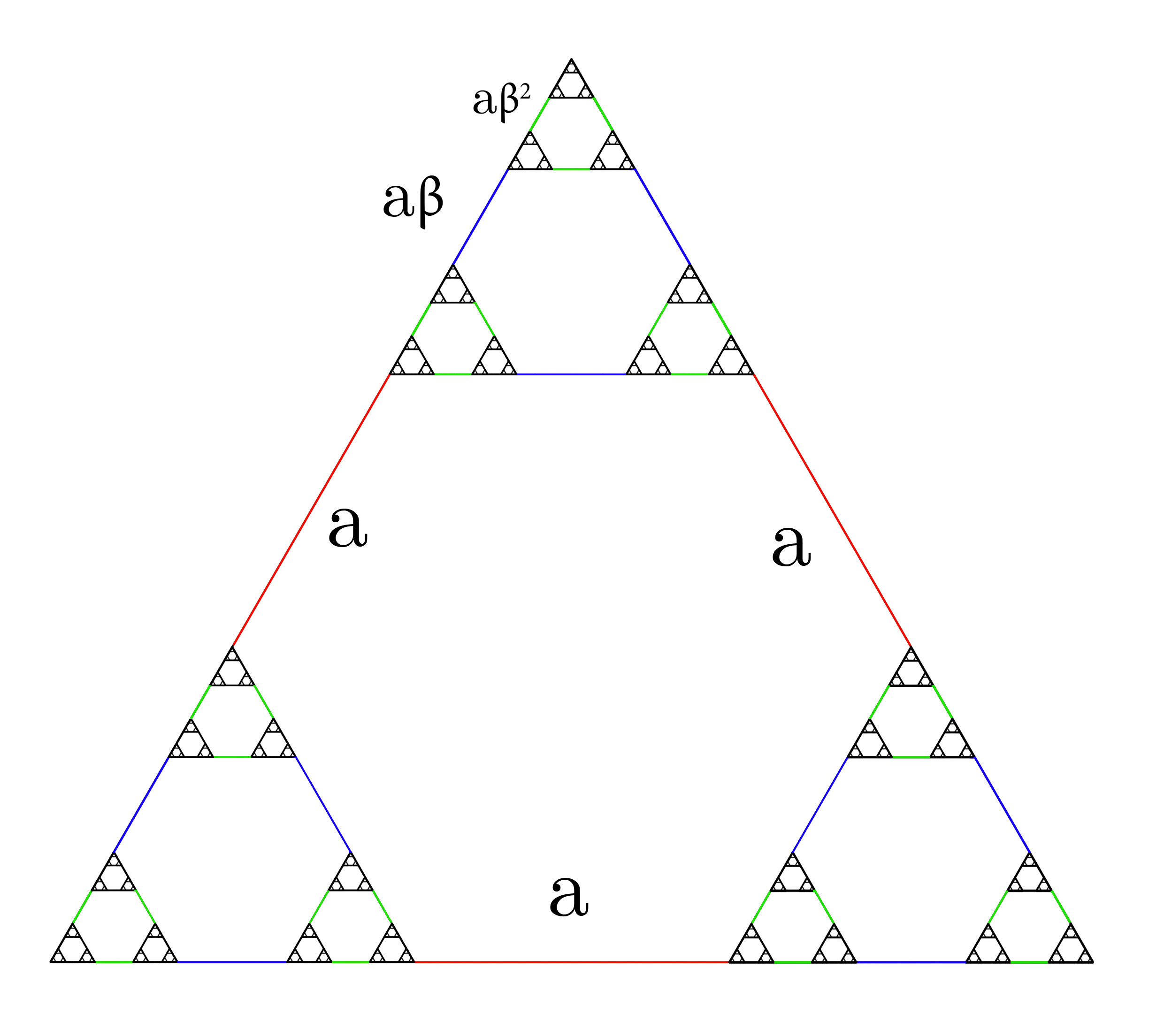}
\caption{Line part of the measure.}\label{hanoimass}
\end{figure}
Now each time the lines get smaller the measure gets scaled by some value $\beta>0$. It is easy to see, that $\beta$ has to be smaller than $\frac 13$ such that this gives us a finite measure. In general we define
\begin{align*}
\mu_l|_{e^w_i}:= a \beta^{|w|} \frac{\lambda^1}{\lambda^1(e^w_i)}
\end{align*}
This means it behaves like the one-dimensional Lebesgue measure on $e^w_i$ but it is normalized and then scaled by $a\beta^{|w|}$. Therefore it doesn't depend on the one-dimensional Lebesgue measure of $e^w_i$. We want $\mu_l$ to be normalized. Therefore we choose $a=\frac 13-\beta$. The construction is illustrated in Figure~\ref{hanoimass}.

Since the closure of the union of all one-dimensional lines is $K$, this measure fulfils the requirements. We can therefore use $\mu_l$ to get Dirichlet forms. The fractal part $\mu_f$ can not be used alone, but we can use any convex combination of these measures.
\begin{align*}
\mu_\eta:=\eta\mu_l+(1-\eta)\mu_f, \ \eta\in (0,1]
\end{align*} 
Note that $\eta=1$ is allowed. In this case we don't have a fractal part. This will be noticable in the spectral asymptotics.\\[.2cm]

With these measures we can define Dirichlet forms and therefore operators on $L^2(K,\mu_\eta)$.
Since $(K,R_\mathcal{R})$ is compact (see \cite[Theorem 7.16]{afk17}) we have the following result with \\ $\mathcal{D}_\mathcal{R}:=\overline{\mathcal{F}_\mathcal{R}\cap C_0(K)}^{\mathcal{E}_{_\mathcal{R},1}^{\frac 12}}=\mathcal{F}_\mathcal{R}$.\\
\begin{lemma} $(\mathcal{E}_\mathcal{R},\mathcal{D}_\mathcal{R})$ is a regular Dirichlet form on $L^2(K,\mu_\eta)$.\label{lem21}\end{lemma}
\begin{proof}
\cite[Lemma 5.1]{haus17}
\end{proof}
Introducing Dirichlet boundary conditions we get another Dirichlet form with $\mathcal{D}_\mathcal{R}^0:=\{u\in\mathcal{D}_\mathcal{R}: \ u|_{V_0}\equiv 0\}$.\\
\begin{lemma} $(\mathcal{E}_\mathcal{R}|_{\mathcal{D}_\mathcal{R}^0\times\mathcal{D}_\mathcal{R}^0}, \mathcal{D}_\mathcal{R}^0)$ is a regular Dirichlet form on $L^2(K,\mu|_{K\backslash V_0})$.\label{lem22}\end{lemma}
\begin{proof}
\cite[Lemma 5.2]{haus17}
\end{proof}
We denote the associated self-adjoint operators with dense domains by $-\Delta_N^{\mu_\eta,\mathcal{R}}$ resp. $-\Delta_D^{\mu_\eta,\mathcal{R}}$. \\
\begin{lemma} $-\Delta_N^{\mu_\eta,\mathcal{R}}$ and $-\Delta_D^{\mu_\eta,\mathcal{R}}$ have discrete non negative spectrum.\label{lem23}\end{lemma}
\begin{proof}
\cite[Lemma 5.3]{haus17}
\end{proof}

Due to Lemma~\ref{lem23} we can write the eigenvalues in nondecreasing order and study the eigenvalue counting functions. Denote by $\lambda_k^{N,\mu_\eta,\mathcal{R}}$ the $k$-th eigenvalue of $-\Delta_N^{\mu_\eta,\mathcal{R}}$ resp. $\lambda_k^{D,\mu_\eta,\mathcal{R}}$ for $-\Delta_D^{\mu_\eta,\mathcal{R}}$ with $k\geq 1$. Now define
\begin{align*}
N_N^{\mu_\eta,\mathcal{R}}(x):=\#\{k\geq 1: \lambda_k^{N,\mu_\eta,\mathcal{R}}\leq x\}\\
N_D^{\mu_\eta,\mathcal{R}}(x):=\#\{k\geq 1: \lambda_k^{D,\mu_\eta,\mathcal{R}}\leq x\}
\end{align*}
Because $\mathcal{D}_\mathcal{R}^0\subset \mathcal{D}_\mathcal{R}$ and $\dim \D_\R/ { \D_{\R}^0}=3$ (since this quotient space are the harmonic functions) we immediately get 
\begin{align*}
N_D^{\mu_\eta,\mathcal{R}}(x)\leq N_N^{\mu_\eta,\mathcal{R}}(x)\leq N_D^{\mu_\eta,\R}(x)+3, \ \forall x\geq 0
\end{align*}
\subsection{Conditions and asymptotic growing}\label{cond}
To be able to calculate the asymptotic growing of the eigenvalue counting functions we need to set some conditions for the sequences of matching pairs. From now on we only consider such sequences $\R=(r_i,\rho_i)_{i\geq1}$ such that there is a $r\in [\frac 13,\frac 35]$ with
\begin{align}
\sum_{i=1}^\infty |r-r_i|<\infty \label{cond1}
\end{align}
If this is satified it is easy to show with the limit comparison test that the series $\sum_{i=1}^\infty |\ln(r^{-1}r_i)|$ converges and thus
\begin{align*}
\prod_{i=1}^\infty r^{-1}r_i\in(0,\infty)
\end{align*}
This means the sequence $a_m:=\prod_{i=1}^m r^{-1}r_i$ is bounded from above and below:
\begin{align*}
\tilde\kappa_1 r^m \leq \delta_m \leq \tilde\kappa_2 r^m, \ \forall m
\end{align*}
For $\delta_m^{(n)}:=r^{n+1}\cdots r^{n+m}=\frac{\delta_{n+m}}{\delta_n}$ this means
\begin{align*} 
\frac{\tilde \kappa_1}{\tilde \kappa_2} r^m \leq \delta^{(n)}_m \leq \frac{\tilde \kappa_2}{\tilde \kappa_1} r^m
\end{align*}
Without loss of generality we can assume that $\tilde \kappa_1\leq 1 \leq \tilde \kappa_2$ and thus with $\kappa_1:=\frac{\tilde \kappa_1}{\tilde \kappa_2} $ and $\kappa_2:=\frac{\tilde \kappa_2}{\tilde \kappa_1} $ we get for all $n$ and $m$: 
\begin{align}
\kappa_1r^m\leq \delta^{(n)}_m\leq \kappa_2r^m \label{cond2}
\end{align}
It is easy to see that (\ref{cond1}) and (\ref{cond2}) are equivalent. However we want to keep (\ref{cond1}) since in the case of $r=\frac 35$ this has a very nice meaning. This is the only case where the fractal part of the quadratic form really exists (see \cite{afk17}).

\begin{theorem}Let $\R$ be a sequence of matching pairs that fulfils the conditions and $\mu=\mu_\eta$ with $\eta\in(0,1]$. Then there exist constants $0<C_1,C_2<\infty$ and $x_0\geq 0$ such that for all $x\geq x_0$:
\begin{align*}
C_1x^{\frac 12 d_S^{\mu,\R}}\leq N_D^{\mu,\R}(x)\leq N_N^{\mu,\R}(x)\leq C_2x^{\frac 12 d_S^{\mu,\R}}
\end{align*}
with 
\begin{align*}
d_S^{\mu,\R}=\begin{cases}
\frac{\ln 9}{\ln 3- \ln r},  \ \text{for } \mu=\mu_\eta \text{ with } \eta \in (0,1)\\[.2cm]
\frac{\ln 9}{- \ln( \beta r)},  \ \text{for } \mu=\mu_1=\mu_l, \ \text{with } \beta>\frac 1{9r}
\end{cases}
\end{align*}\label{theo24}
\end{theorem}
\begin{proof}
The proof for Theorem~\ref{theo24} can be found in \cite[Theorem 6.1]{haus17} for the case of $\eta=\frac 12$ and slightly stronger conditions. The generalization to $\mu_\eta$ with $\eta\in (0,1]$ is straightforward and the generalization to the weaker conditions in this work will appear in a subsequent publication \cite{hausgen}.
\end{proof}
For sake of notation, we omit the dependencies $\mu_\eta$ and $\R$ in the notation for $d_S$. We can also express Theorem~\ref{theo24} in the following form:
\begin{align*}
0<\liminf_{x\rightarrow \infty} N_\ast^{\mu_\eta,\R}(x)x^{-\frac 12d_S}\leq \limsup_{x\rightarrow \infty} N_\ast^{\mu_\eta,\R}(x)x^{-\frac 12d_S}<\infty, \quad \ast= N,D
\end{align*}

\section{Review on oscillations on SG}\label{chapsirp}
Before we start to look for oscillations in the leading term for the Stretched Sierpinski Gasket we quickly want to review how we can tackle this problem for the usual Sierpinski Gasket. \\

The Sierpinski Gasket $S$ (see Figure~\ref{sierpinski}) is the attractor of the IFS $(F_1,F_2,F_3)$ consisting of three similitudes with contraction ratios $\frac 12$. 
\begin{figure}[H]
\centering
\includegraphics[width=0.6\textwidth]{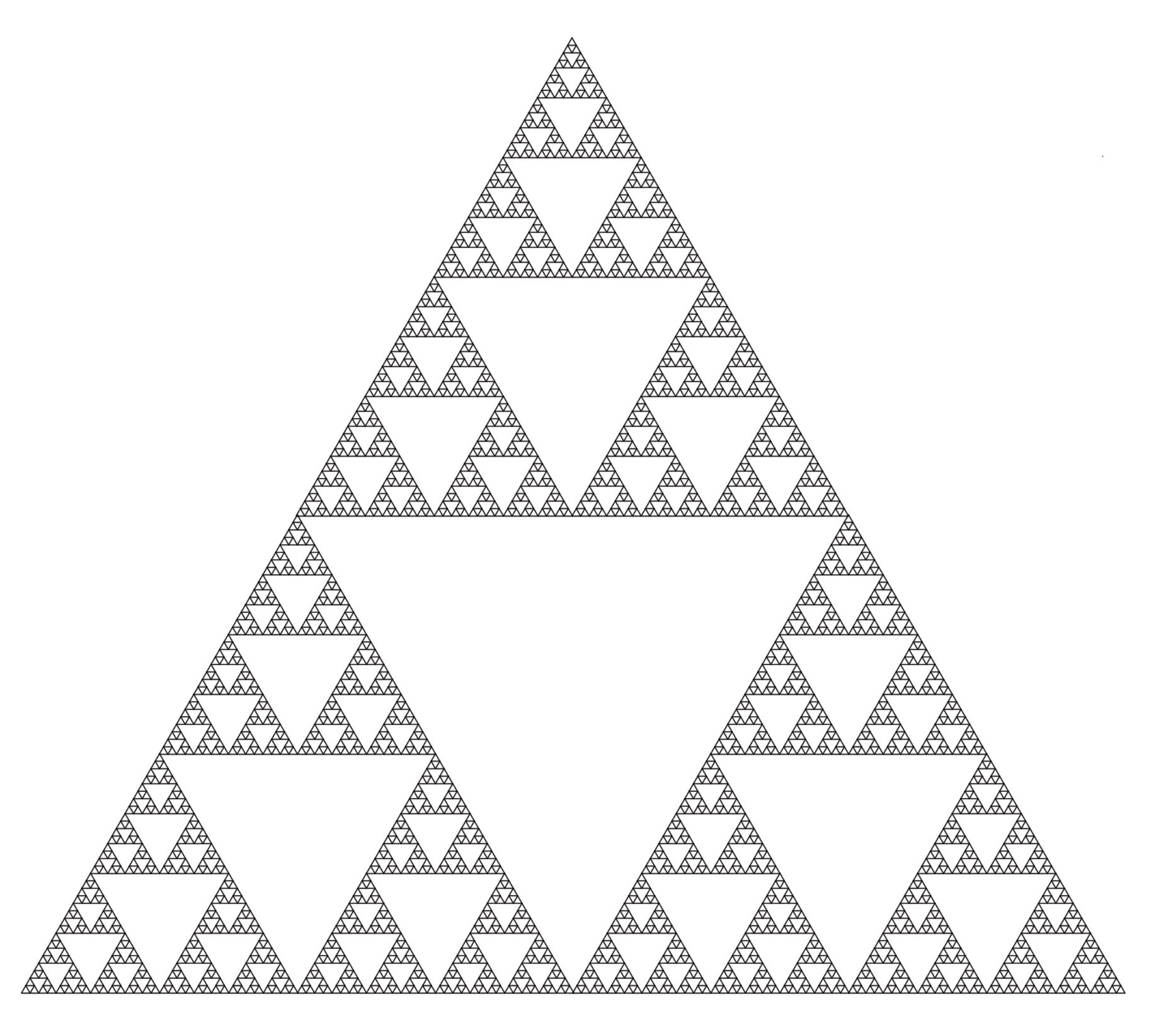}
\caption{The Sierpinski Gasket.}\label{sierpinski}
\end{figure}
On the Sierpinski Gasket we have the following useful rescaling property of the symmetric energy:
\begin{align*}
\E(u,v)=\sum_{i=1}^3 \left(\frac 35\right)^{-1} \E(u\circ F_i,v\circ F_i)
\end{align*}
If we use the normalized Hausdorff measure to get a Dirichlet form, it is well known that the eigenvalue counting functions (either Dirichlet or Neumann) have the asymptotic growing with spectral exponent $d_S=\frac{\ln 9}{\ln 5}$ (see \cite{fushim},\cite{kl93}). This means
\begin{align*}
0<\liminf_{x\rightarrow \infty} N_D(x)x^{-\frac 12d_S}\leq \limsup_{x\rightarrow \infty} N_D(x)x^{-\frac 12d_S} <\infty
\end{align*}
We are in particular interested in the $\leq$ sign. Do the two limits coincide or do they differ? To answer this we introduce Dirichlet and Neumann boundary conditions at $V_1\backslash V_0$ to get the following inequalities (see \cite{kl93}):
\begin{align*}
\sum_{i=1}^3 N_D(\tfrac 15 x)\leq N_D(x)\leq N_N(x)\leq \sum_{i=1}^3N_N(\tfrac 15 x)\leq \sum_{i=1}^3N_D(\tfrac 15 x)+9
\end{align*}
Now we have the Dirichlet eigenvalue counting function wedged in between some scaled version of itself. The scaling $\frac 15$ consists of two factors. One is the scaling of the energy $\frac 35$ and the other is the scaling of the measure $\frac 13$.\\

We can use these inequalities to apply renewal theory (see \cite{kl93} and \cite{kig98}) and get a positive periodic function $G$ with period $\frac 12\ln 5$ such that
\begin{align*}
N_D(x)=G\left(\frac{\ln x}2\right)x^{\frac 12 d_S}+\mathcal{O}(1)
\end{align*}
The boundedness of the error term is strongly connected to the boundedness of the error term in the inequalities.\\

However this doesn't answer the question of convergence since the periodic function $G$ could still be constant.  To get an answer we look at localized eigenfunctions. These are eigenfunctions that are supported on a proper subset of the Sierpinski Gasket. The existence of such eigenfunctions is strongly connected to the existence of so called Dirichlet-Neumann eigenfunctions. These are eigenfunctions that are simultaneously eigenfunctions to Dirichlet and Neumann boundary conditions. They are also called pre-localized eigenfunctions. As it turns out, such eigenfunctions exist on the Sierpinski Gasket (see \cite{bk97}). Let $u$ be such a pre-localized eigenfunction with eigenvalue $\lambda$. We then define for $w\in\A^n$
\begin{align*}
u_w(x):=\begin{cases} u\circ F_w^{-1}(x) &,  x\in S_w\\ 0&, \text{otherwise} \end{cases} 
\end{align*}
where $S_w=F_w(S)$ is a $n$-cell of the Sierpinski Gasket. Using the rescaling properties of the energy and the measure we can show that $u_w$ itself is again an eigenfunction (Dirichlet and Neumann) with eigenvalue $\lambda 5^n$ (see \cite{bk97}). We can do this construction for all $n$-cells and there are $3^n$ many. This means the eigenvalue $\lambda 5^n$ has multiplicity at least $3^n$. Therefore we have a sequence of growing eigenvalues with very high multiplicites. High multiplicities lead to very big jumps in the eigenvalue counting function. This sequence of eigenvalues is enough to show, that $\liminf$ and $\limsup$ can't coincide (see \cite{bk97} or compare to chapter~\ref{chaple}).

\section{How does the measure scale?}\label{chapmeas}
How do our measures scale for smaller getting cells? For the fractal part this is clear due to its definition:
\begin{align*}
\mu_f(K_w)=\frac 1{3^{|w|}}
\end{align*}
The measure $\mu_l$ on the line part exhibits another scaling. Since
\begin{align*}
1&=3^{|w|}\mu_l(K_w)+\sum_{\tilde w: |\tilde w|< |w|,i} \mu_l(e_i^{\tilde w})\\
&=3^{|w|}\mu_l(K_w)+3a\sum_{k=0}^{|w|-1}(3\beta)^k\\
&=3^{|w|}\mu_l(K_w)+(1-(3\beta)^{|w|})\\[0.2cm]
\Rightarrow \mu_l(K_w)&=\beta^{|w|}
\end{align*}
For $\mu_\eta$ with $\eta\in(0,1)$ we get the following estimates which will be useful later on.
\begin{align*}
(1-\eta)\left(\frac 13\right)^{|w|}\leq \mu_\eta(K_w)\leq \left(\frac 13\right)^{|w|}
\end{align*}

Let $\mu=\mu_\eta$ for any $\eta\in(0,1]$. In the proof of Theorem~\ref{theo24} in \cite{haus17} we used a set of measures $\mu^w_\eta$. These measures were defined as follows. For $w\in \A^n$ define
\begin{align*}
\mu_\eta^w:=\mu_\eta(K_w)^{-1}\mu_\eta\circ G_w
\end{align*}
$\mu^w_\eta$ is a measure on all of $K$ but only reflects the properties of $\mu_\eta$ in $K_w$. For self-similar measures $\nu$ we have $\nu^w=\nu$ but we are not in that case. The two parts $\mu_l$ and $\mu_f$ both fulfil the equality on their own but not the sum of them.
\begin{align*}
\mu_l^w=\mu_l \text{ and } \mu_f^w=\mu_f \text{ but }\mu_\eta^w\neq \mu_\eta \text{ for } \eta\in(0,1)
\end{align*}
This is the case since the scaling constants $\mu_l(K_w)$ and $\mu_f(K_w)$ are not equal, i.e. $\beta\neq \frac 13$. So equality is asked too much, but they still have a lot in common. 
\begin{align*}
\mu_\eta^w&=\mu_\eta(K_w)^{-1}\mu_\eta\circ G_w\\
&=\mu_\eta(K_w)^{-1}\left((\eta\mu_l+(1-\eta)\mu_f)\circ G_w\right)\\
&=\frac{\eta}{\mu_\eta(K_w)}\mu_l\circ G_w+ \frac{1-\eta}{\mu_\eta(K_w)}\mu_f\circ G_w\\
&=\frac{\eta\mu_l(K_w)}{\mu_\eta(K_w)}\mu_l^w+ \frac{(1-\eta)\mu_f(K_w)}{\mu_\eta(K_w)}\mu_f^w\\
&=\underbrace{\frac{\eta\mu_l(K_w)}{\mu_\eta(K_w)}}_{=:\eta^w}\mu_l+ \frac{(1-\eta)\mu_f(K_w)}{\mu_\eta(K_w)}\mu_f\\
&=\eta^w\mu_l+(1-\eta^w)\mu_f
\end{align*}
This means we get the following result:
\begin{lemma}\label{lemmeas}
For $\eta\in(0,1)$ with $\eta^w:=\frac{\eta\mu_l(K_w)}{\mu_\eta(K_w)}$ we get
\begin{align*}
\mu_\eta^w&=\mu_\eta(K_w)^{-1} \mu_\eta\circ G_w\\
&=\mu_{\eta^w}\\
&=\mu_{\eta^{|w|}}
\end{align*}
\end{lemma}
Thus $\mu_\eta^w$ is again a convex combination of the two parts of the measure but with different parameter $\eta^w$. The last equality indicates, that the measures $\mu_\eta^w$ are all the same for words of same length. This is due to the fact, that $\mu_\eta$ is very symmetric and behaves the same on all $n$-cells. So we could also write 
\begin{align*}
\mu_\eta^{(n)}:=\mu_\eta^w, \ \text{for any } w\in \A^n
\end{align*}
We also get estimates for the $L^2$ norms. We have 
\begin{align*}
||f||_{\mu_\eta^w}^2&=\int_K f^2 d\mu_\eta^w\\
&=\eta^w \int_K f^2 d\mu_l + (1-\eta^w)\int_K f^2 d\mu_f\\[.1cm]
||f||_{\mu_\eta}^2&=\eta \int_K f^2 d\mu_l+(1-\eta) \int_K f^2 d\mu_f
\end{align*}
Now since $\mu_l(K_w)\leq \mu_f(K_w)$ we get
\begin{align*}
\frac{\mu_l(K_w)}{\mu_\eta(K_w)}\cdot ||f||_{\mu_\eta}^2\leq ||f||_{\mu_\eta^w}^2\leq \frac{\mu_f(K_w)}{\mu_\eta(K_w)}\cdot ||f||_{\mu_\eta}^2
\end{align*}
for all $f\in L^2(K,\mu_\eta)$. It is easy to see that those estimates are sharp.
\\[.2cm]

\section{How does the energy scale and why isn't renewal theory applicable?}\label{chapper}
In this chapter we want to look at the energy and see how it rescales. We want to get similar estimates for the eigenvalue counting functions as for the Sierpinski Gasket in chapter~\ref{chapsirp}. The eigenvalue counting function has many dependencies which we all want to include in the notation whenever it is necessary. This means for a Dirichlet form $\E$ with domain $\D$ in the Hilbert space $L^2(K,\mu)$ we denote the eigenvalue counting function of the associated self-adjoint operator at point $x$ with
\begin{align*}
N(\E,\D,\mu,x)
\end{align*} 

From \cite{afk17} we know, that the energy has the following rescaling property. With $\R^{(n)}=(r_{n+k},\rho_{n+k})_{k\geq 1}$ we have for $u,v\in \D_\R$
\begin{align*}
\E_\R(u,v)=\sum_{w\in\A^n}\frac 1{\delta_n} \E_{\R^{(n)}}(u\circ G_w,v\circ G_w) + \sum_{k=1}^n \frac 1{\gamma_k} \D_k^I(u,v)
\end{align*}
This rescaling property is similar to the one on the self-similar Sierpinski Gasket but not quite as nice. We see that there is an additional term on the right hand site. This comes from the one-dimensional lines that connect the $n$-cells. As it turns out we will be able to work with this as it is somehow of lower order. The other difference is that the quadratic forms $\E_\R$ on the left and $\E_{\R^{(n)}}$ on the right hand side differ. As the sequences of matching pairs are different ones, the quadratic forms are different. This also doesn't appear in the self-similar case. Nonetheless we try to get the same results concerning periodicity. We introduce Neumann boundary conditions at $V_n\backslash V_0$ as in the proof of Theorem~\ref{theo24} (see \cite{haus17}).
 With 
\begin{align*}
\mathcal{D}_{\R,K_w}:=\{u\in L^2(K,\mu_\eta): u|_{K_w^c}\equiv 0, \exists f\in \mathcal{D}_\R: f|_{K_w}=u\}
\end{align*}
we have\\[.2cm]
\begin{lemma}
\begin{align*}
N(\mathcal{E}_{\mathcal{R}},\mathcal{D}_{\R,K_w},\mu_\eta,x)=N(\mathcal{E}_{\mathcal{R}^{(n)}},\mathcal{D}_{\mathcal{R}^{(n)}},\mu_\eta^w,\mu_\eta(K_w)\delta_n x)
\end{align*}
\end{lemma}
\begin{proof} Let $u$ be an eigenfunction of $(\mathcal{E}_\mathcal{R},\mathcal{D}_{\R,K_w},\mu_\eta)$ with eigenvalue $\lambda$. That means for all $v\in \mathcal{D}_{\R,K_w}$ we have
\begin{align*}
\mathcal{E}_{\mathcal{R}}(u,v)=\lambda(u,v)_{\mu_\eta} 
\end{align*}
Since $\mathcal{E}_\mathcal{R}(u,v)=\frac 1{\delta_n}\mathcal{E}_{\mathcal{R}^{(n)}}(u\circ G_w,v\circ G_w)$ and
\begin{align*}
\lambda(u,v)_{\mu_\eta}&=\lambda\int_K uvd{\mu_\eta} \\
&=\lambda\int_{K_w}uvd{\mu_\eta} \\
&=\lambda {\mu_\eta}(K_w)\int_K u\circ G_w\cdot v\circ G_w d\mu_\eta^w
\end{align*}
so
\begin{align*}
\mathcal{E}_{\mathcal{R}^{(n)}}(u\circ G_w,v\circ G_w)=\lambda\delta_n\mu_\eta(K_w)(u\circ G_w,v\circ G_w)_{\mu_\eta^w}
\end{align*}
This holds for all $v\in \mathcal{D}_{\R,K_w}$, but due to the construction of $\D_{\R,K_w}$ we can reach every $\tilde v\in \D_\R$ with $v\circ G_w$. Therefore $u\circ G_w$ is an eigenfunction of $(\mathcal{E}_{\mathcal{R}^{(n)}},\mathcal{D}_{\mathcal{R}^{(n)}},\mu_\eta)$ with eigenvalue $\lambda\mu_\eta(K_w)\delta_n$. \\

The other direction works analogously with $\tilde u:=u\circ G_w^{-1}$ as an eigenfunction of $(\mathcal{E}_\R,\mathcal{D}_{\R,K_w},\mu_\eta)$ if $u$ is an eigenfunction of ($\mathcal{E}_{\mathcal{R}^{(n)}},\mathcal{D}_{\mathcal{R}^{(n)}},\mu_\eta^w)$. We therefore have the desired result.
\end{proof}

With an analogous proof we get the same result for the Dirichlet case:
\begin{lemma} 
\begin{align*}
N(\mathcal{E}_{\mathcal{R}},\mathcal{D}_{\R,K_w}^0,\mu_\eta,x)=N(\mathcal{E}_{\mathcal{R}^{(n)}},\mathcal{D}_{\mathcal{R}^{(n)}}^0,\mu_\eta^w,\mu_\eta(K_w)\delta_n x)
\end{align*}
\end{lemma}
There are many differences to the non-stretched case. We have different forms, domains and measures. To be able to apply renewal theory we need to get rid of these differences. We already have estimates for the measures. Now we need estimates on the quadratic forms. To be able to do this we need to tighten our conditions.\\

To show that strict periodic behaviour like on the Sierpinski Gasket is very unlikely we introduce slightly stricter conditions than in chapter~\ref{cond} solely for this chapter. For now we only consider sequences of matching pairs that fulfil the conditions of chapter~\ref{cond} with $r<\frac 35$. This excludes the highest possible value of $r$ but this chapter has the purpose of showing where the problems lie. So it suffices to consider these stronger conditions.\\

As a quick reminder the quadratic forms are of the following structure:
\begin{align*}
\mathcal{E}_{\mathcal{R}}(u,u)&=\lim\limits_{k\rightarrow \infty} \frac 1{\delta_k} Q_k(u,u)+\sum_{k\geq 1} \frac 1{\gamma_k} \mathcal{D}_k^I(u,u)\\
&\text{with } \ \delta_k=r_1\cdots r_k \ \text{ and } \ \gamma_k=r_1\cdots r_{k-1}\rho_k\\[.2cm]
\mathcal{E}_{\mathcal{R}^{(n)}}(u,u)&=\lim\limits_{k\rightarrow \infty} \frac 1{\delta_k^{(n)}} Q_k(u,u)+\sum_{k\geq 1} \frac 1{\gamma_k^{(n)}} \mathcal{D}_k^I(u,u)\\
&\text{with } \ \delta_k^{(n)}=r_{n+1}\cdots r_{n+k} \ \text{ and } \ \gamma_k^{(n)}=r_{n+1}\cdots r_{n+k-1}\rho_{n+k}\\
\end{align*}
To compare these forms we need estimates between $\delta_k$ and $\delta^{(n)}_k$, and also between $\gamma_k$ and $\gamma_k^{(n)}$. We have
\begin{align*}
\delta_k^{(n)}&=\delta_k \frac{r_{k+1}\cdots r_{k+n}}{\delta_n}
\end{align*}
With the estimates from chapter~\ref{cond} we get
\begin{align*}
\kappa_1\frac{r^n}{\delta_n}\delta_k\leq \delta_k^{(n)}\leq \kappa_2 \frac{r^n}{\delta_n}\delta_k
\end{align*}
For the $\gamma_k$:
\begin{align*}
\gamma_k^{(n)}&=\frac{\gamma_{n+k}}{\delta_n}\\
&=\gamma_k \frac{r_k\cdots r_{n+k-1}}{\delta_n}\frac{\rho_{n+k}}{\rho_k}
\end{align*}
Again with the estimates from chapter~\ref{cond} we have
\begin{align*}
\kappa_1 \frac{r^n}{\delta_n} \frac{\rho_{n+k}}{\rho_k} \gamma_k\leq \gamma_k^{(n)}\leq \kappa_2 \frac{r^n}{\delta_n}\frac{\rho_{n+k}}{\rho_k} \gamma_k
\end{align*}
With the stronger conditions we introduced for this chapter we get constants $\kappa_3\leq \frac{\rho_{n+k}}{\rho_k}\leq \kappa_4$ for all $n,k$ and thus all in all we get constants $c_1$ and $c_2$ such that
\begin{align*}
c_1\frac{r^n}{\delta_n}\E_{\R^{(n)}}(u,u)\leq \E_\R(u,u)\leq c_2\frac{r^n}{\delta_n} \E_{\R^{(n)}}(u,u)
\end{align*}
This also immediately implies that the domains coincide:
\begin{align*}
\D_\R=\D_{\R^{(n)}}, \ \text{ for all } n
\end{align*}
With these estimates we can get estimates on the eigenvalue counting functions.
We want to compare the eigenvalues and eigenvalue counting functions of $(\mathcal{E}_{\mathcal{R}},\D_\R,\mu_\eta)$ and $(\mathcal{E}_{\mathcal{R}^{(n)}},\D_{\R^{(n)}},\mu_\eta^w)$. To do this we use the Max-Min Principle (see \cite[Theorem 2, Chapter 10]{bs87}):
\begin{align*}
\lambda_k(\mathcal{E},\D,\nu)=\max_{\Phi\subset \mathcal{D}} \inf_{\substack{u\in \Phi\\ ||u||_\nu=1}} \mathcal{E}(u,u)
\end{align*}
where the maximum is taken over all subspaces $\Phi$ with co-dimension equal or less than $k-1$. The eigenvalues depend on the resistance form as well as the measure $\nu$. The occuring norm $||u||_\nu$ is the $L^2$ norm in $L^2(K,\nu)$. The condition in the infimum can be changed to $||u||_\nu\geq 1$ since it takes its lowest value at $||u||_\nu=1$. 
We make use of the following two estimates:
\begin{align*}
c_1\frac{r^n}{\delta_n}\mathcal{E}_{\mathcal{R}^{(n)}}(u,u)&\leq \mathcal{E}_\mathcal{R}(u,u)\leq c_2\frac{r^n}{\delta_n}\mathcal{E}_{\mathcal{R}^{(n)}}(u,u)\\
\sqrt{\frac{\mu_l(K_w)}{\mu_\eta(K_w)}}||f||_{\mu_\eta}&\leq ||f||_{\mu_\eta^w}\leq \sqrt{\frac{\mu_f(K_w)}{\mu_\eta(K_w)}} ||f||_{\mu_\eta}
\end{align*}
The maximum doesn't change since $\mathcal{D}_\mathcal{R}=\mathcal{D}_{\mathcal{R}^{(n)}}$. However we have different Hilbert-spaces $L^2(K,{\mu_\eta})$ and $L^2(K,\mu_\eta^w)$, that means the norm changes.
\begin{align*}
\lambda_k(\mathcal{E}_\mathcal{R},\D_\R,\mu_\eta)&=\max_{\substack{\Phi\subset \D_\R \\ \dim \D_\R /\Phi\leq k-1}} \inf_{\substack{u\in \Phi\\ ||u||_{\mu_\eta}\geq1}} \mathcal{E}_{\mathcal{R}}(u,u)\\
&\leq\max_{\substack{\Phi\subset \D_\R \\ \dim \D_\R /\Phi\leq k-1}} \inf_{\substack{u\in \Phi\\ ||u||_{\mu_\eta}\geq1}} c_2\frac{r^n}{\delta_n}\mathcal{E}_{\mathcal{R}^{(n)}}(u,u)
\end{align*}
We saw $||u||_{{\mu_\eta}^w}\leq \sqrt{\frac{\mu_f(K_w)}{\mu_\eta(K_w)}} ||u||_{\mu_\eta}$, that means the condition $||u||_{\mu_\eta^w}\geq \sqrt{\tfrac{\mu_f(K_w)}{\mu_\eta(K_w)}}$ is stronger than $||u||_{\mu_\eta}\geq 1$ and therefore the set over which the infimum is taken gets smaller and thus the infimum gets bigger.
\begin{align*}
\Rightarrow \lambda_k(\mathcal{E}_\mathcal{R},\D_\R,\mu_\eta)&\leq  c_2\frac{r^n}{\delta_n}\max_{\substack{\Phi\subset \mathcal{D}_\R\\ \dim \mathcal{D}_\R /\Phi\leq k-1}} \inf_{\substack{u\in \Phi\\ ||u||_{\mu_\eta^w}\geq \sqrt{\frac{\mu_f(K_w)}{\mu_\eta(K_w)}}}} \mathcal{E}_{\mathcal{R}^{(n)}}(u,u)\\
&= c_2\frac{r^n}{\delta_n}\max_{\substack{\Phi\subset \mathcal{D}_\R\\ \dim \mathcal{D}_\R /\Phi\leq k-1}} \inf_{\substack{u\in \Phi\\ ||u||_{\mu_\eta^w}\geq 1}}  \mathcal{E}_{\mathcal{R}^{(n)}}\left(\sqrt{\frac{\mu_f(K_w)}{\mu_\eta(K_w)}}u,\sqrt{\frac{\mu_f(K_w)}{\mu_\eta(K_w)}}u\right)\\
&=c_2\frac{r^n}{\delta_n}\frac{\mu_f(K_w)}{\mu_\eta(K_w)}\max_{\substack{\Phi\subset \mathcal{D}_\R\\ \dim \mathcal{D}_\R /\Phi\leq k-1}} \inf_{\substack{u\in \Phi\\ ||u||_{\mu_\eta^w}\geq 1}}  \mathcal{E}_{\mathcal{R}^{(n)}}(u,u)\\
&=c_2\frac{r^n}{\delta_n}\frac{\mu_f(K_w)}{\mu_\eta(K_w)}\lambda_k(\mathcal{E}_{\mathcal{R}^{(n)}},\D_\R,\mu_\eta^w)
\end{align*}
The other direction works the same, so we get:
\begin{align*}
c_1\tfrac{r^n}{\delta_n}\tfrac{\mu_l(K_w)}{\mu_\eta(K_w)}\lambda_k(\mathcal{E}_{\mathcal{R}^{(n)}},\D_\R,\mu_\eta^w)\leq \lambda_k(\mathcal{E}_{\mathcal{R}},\D_\R,\mu_\eta)\leq c_2\tfrac{r^n}{\delta_n}\tfrac{\mu_f(K_w)}{\mu_\eta(K_w)}\lambda_k(\mathcal{E}_{\mathcal{R}^{(n)}},\D_\R,\mu_\eta^w)
\end{align*}
For the eigenvalue counting function this means:
\begin{align*}
N(\mathcal{E}_{\mathcal{R}^{(n)}},\D_\R,\mu_\eta^w, \tfrac x{c_2\frac{r^n}{\delta_n}\frac{\mu_f(K_w)}{\mu_\eta(K_w)}})\leq N(\mathcal{E}_\mathcal{R},\D_\R,\mu_\eta,x)\leq N(\mathcal{E}_{\mathcal{R}^{(n)}},\D_\R,\mu_\eta^w, \tfrac x{c_1\frac{r^n}{\delta_n}\frac{\mu_l(K_w)}{\mu_\eta(K_w)}})
\end{align*}
If we change $\mathcal{D}_\R$ to $\mathcal{D}_\R^0$ we get the same results for the Dirichlet eigenvalues and counting function.\\

Adding this to the estimates from before leads to the following estimates for the eigenvalue counting functions for all $x\geq 0$ with 
\begin{align*}
\D_{\R,J_n}:=\{u\in L^2(K,\mu_\eta): u|_{J_n^c}\equiv 0, \exists f\in \mathcal{D}_\R: f|_{J_n}=u\}
\end{align*}
\begin{align*}
N^{\mu_\eta,\R}_D(x)&\leq N^{\mu_\eta,\R}_N(x)\\
&\leq \sum_{w\in\mathcal{A}^n}N(\mathcal{E}_\mathcal{R},\mathcal{D}_{\R,K_w},\mu_\eta,x)+N(\mathcal{E}_\R,\D_{\R,J_n},\mu_\eta,x)\\
&= \sum_{w\in\mathcal{A}^n} N(\mathcal{E}_{\mathcal{R}^{(n)}},\mathcal{D}_{\R^{(n)}},\mu_\eta^w,\mu_\eta(K_w)\delta_nx)+N(\mathcal{E}_\R,\D_{\R,J_n},\mu_\eta,x)\\
&\leq  \sum_{w\in\mathcal{A}^n} N(\mathcal{E}_{\mathcal{R}},\mathcal{D}_\R,\mu_\eta,\mu_\eta(K_w)\delta_n x c_2\frac{r^n}{\delta_n}\frac{\mu_f(K_w)}{\mu_\eta(K_w)})+N(\mathcal{E}_\R,\D_{\R,J_n},\mu_\eta,x)\\
&=  \sum_{w\in\mathcal{A}^n} N_N^{\mu_\eta,\R}(c_2 r^n \mu_f(K_w) x)+N(\mathcal{E}_\R,\D_{\R,J_n},\mu_\eta,x)\\
&\leq \sum_{w\in\mathcal{A}^n}( N_D^{\mu_\eta,\R}(c_2 r^n \mu_f(K_w) x) +3)+N(\mathcal{E}_\R,\D_{\R,J_n},\mu_\eta,x)
\end{align*}
The lower bound works analogously so all in all we get for $\eta\in(0,1)$
\begin{align*}
\sum_{w\in\mathcal{A}^n} N_D^{\mu_\eta,\R}(c_1 (\beta r)^n x)&\leq N_D^{\mu_\eta,\R}(x)\leq N_N^{\mu_\eta,\R}(x)\\&\leq \sum_{w\in\mathcal{A}^n} N_D^{\mu_\eta,\R}(c_2 (\tfrac r3)^n x)+ 3^{n+1}+N(\mathcal{E}_\R,\D_{\R,J_n},\mu_\eta,x)
\end{align*}
If we only use the line part of the measure $\mu_l$ we don't need the estimates on the $L^2$ norms. This leads to the following estimates
\begin{align*}
\sum_{w\in\mathcal{A}^n} N_D^{\mu_\eta,\R}(c_1 (\beta r)^n x)&\leq N_D^{\mu_\eta,\R}(x)\leq N_N^{\mu_\eta,\R}(x)\\&\leq \sum_{w\in\mathcal{A}^n} N_D^{\mu_\eta,\R}(c_2 (\beta r)^n x)+ 3^{n+1}+N(\mathcal{E}_\R,\D_{\R,J_n},\mu_\eta,x)
\end{align*}
with constants $c_1$ and $c_2$ for all $n$. \\

So for $\eta\in(0,1)$ we don't even get the same scaling. We have $\beta r$ for the lower and $\tfrac r3$ for the upper bound. This is due to the two different scalings of the measure. If we only use the line part of the measure we have the same scaling $\beta r$ on both sides but we still have the constants $c_1$ and $c_2$. In general they do not coincide, even not asymptotically for $n\rightarrow \infty$. We can't apply renewal theory here since we need the exact same value on both sides. \\

Furthermore since we can't get rid of the constants it is very unlikely to get strict periodic behaviour even with other methods. We lack some kind of symmetry which is necessary to get this very strict periodic behaviour in general. There are some special cases where we are able to get strict periodicity and we will handle them in chapter~\ref{special}.

\section{Existence of localized eigenfunctions and proof of non-convergence}\label{chaple}
So we saw that strict periodic behaviour as for the Sierpinski Gasket is not very likely. However we still want to answer the question of convergence. This means even if there is no strict periodic behaviour, there could still be oscillations. To answer this question we look at localized eigenfunctions just as for the Sierpinksi Gasket. As it turns out there still are localized eigenfunctions on the Stretched Sierpinski Gasket. To show this we first show the existence of so called Dirichlet-Neumann eigenfunctions. These are eigenfunctions that are simultaneously eigenfunctions of $-\Delta_D^{\mu,\R}$ as well as $-\Delta_N^{\mu,\R}$. Since we can construct localized eigenfunctions by the use of them we also call them pre-localized eigenfunctions.

\begin{lemma}
Let $\R$ be any sequence of matching pairs and $\eta\in (0,1]$. Then there exists a Dirichlet-Neumann eigenfunction $u$ with eigenvalue $\lambda$ of $(\E_\R,\D_\R,\mu_\eta)$. This means $\exists u\in \D_{\R}^0$ with $\E_\R(u,v)=\lambda(u,v)_{\mu_\eta}$, $\forall v\in \D_\R$. \label{lemdn}
\end{lemma}

\begin{proof}
The proof of the existence of localized eigenfunctions follows the arguments in \cite{bk97} where the existence was shown for p.c.f. self-similar sets under certain conditions on the symmetry of the set. However the Stretched Sierpinski Gasket is not self-similar but the strong symmetry suffices to apply the ideas. We modify the ideas slightly to get more information about the eigenvalue.\\

On the Stretched Sierpinski Gasket we have the following symmetries, which are the same as for the Sierpinski Gasket. These symmetries are fulfiled by the geometry of the set, the resistance forms and the measures $\mu_\eta$.
\begin{itemize}
\item Three Rotations: $0^\circ$, $120^\circ$ and $240^\circ$
\item Three Reflections: One on each bisecting line
\end{itemize}
By $\sigma$ we denote the $120^\circ$ rotation and by $\tau$ the reflection at the bisecting line trough $p_1$ as you can see in Figure~\ref{symmetries}.
\begin{figure}[H]
\centering
\includegraphics[width=0.63\textwidth]{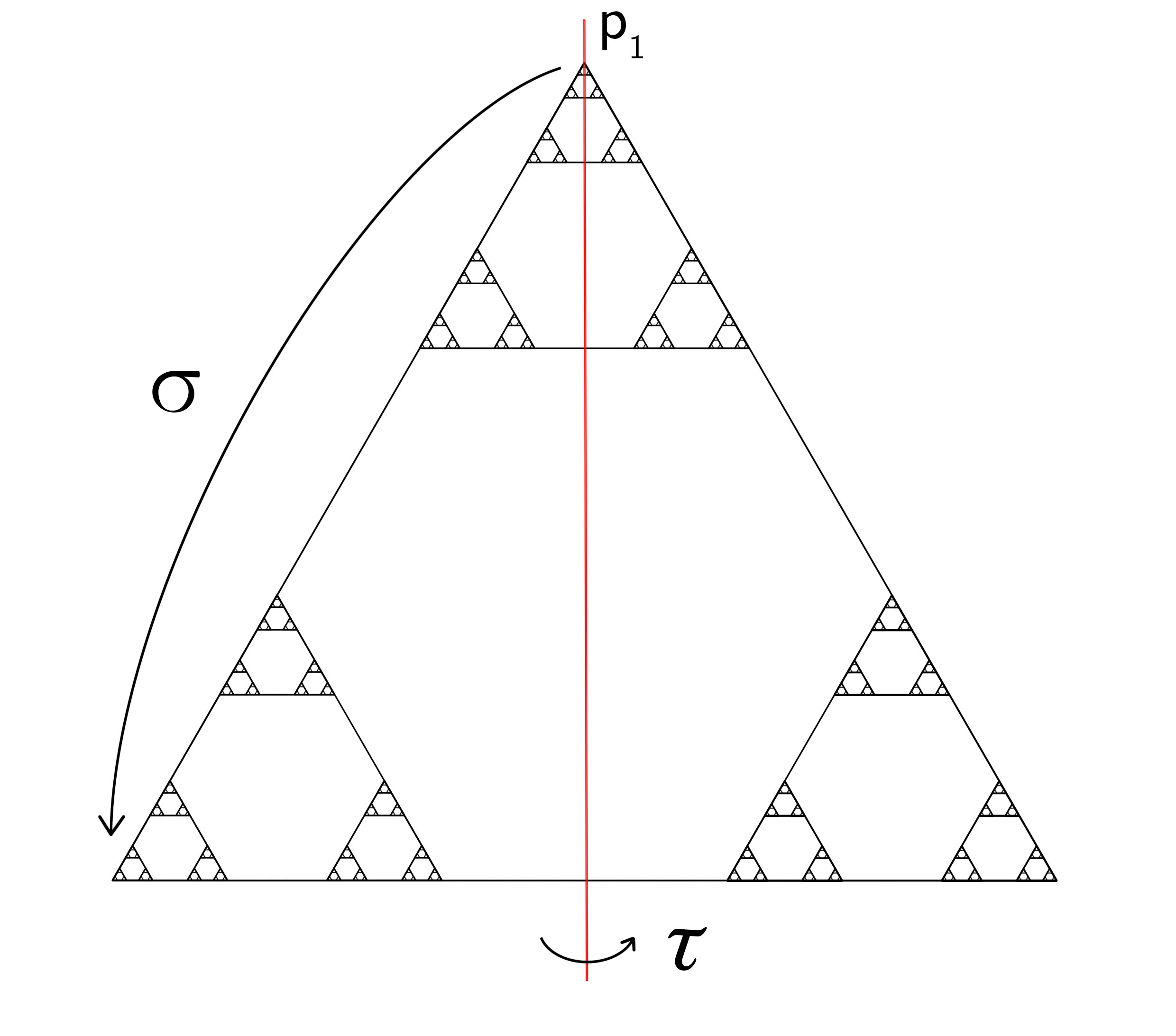}
\caption{Symmetries of the SSG.}\label{symmetries}
\end{figure}
We divide $K$ into six parts in the following manner illustrated in Figure~\ref{ssg6}.
\begin{figure}[H]
\centering
\includegraphics[width=0.6\textwidth]{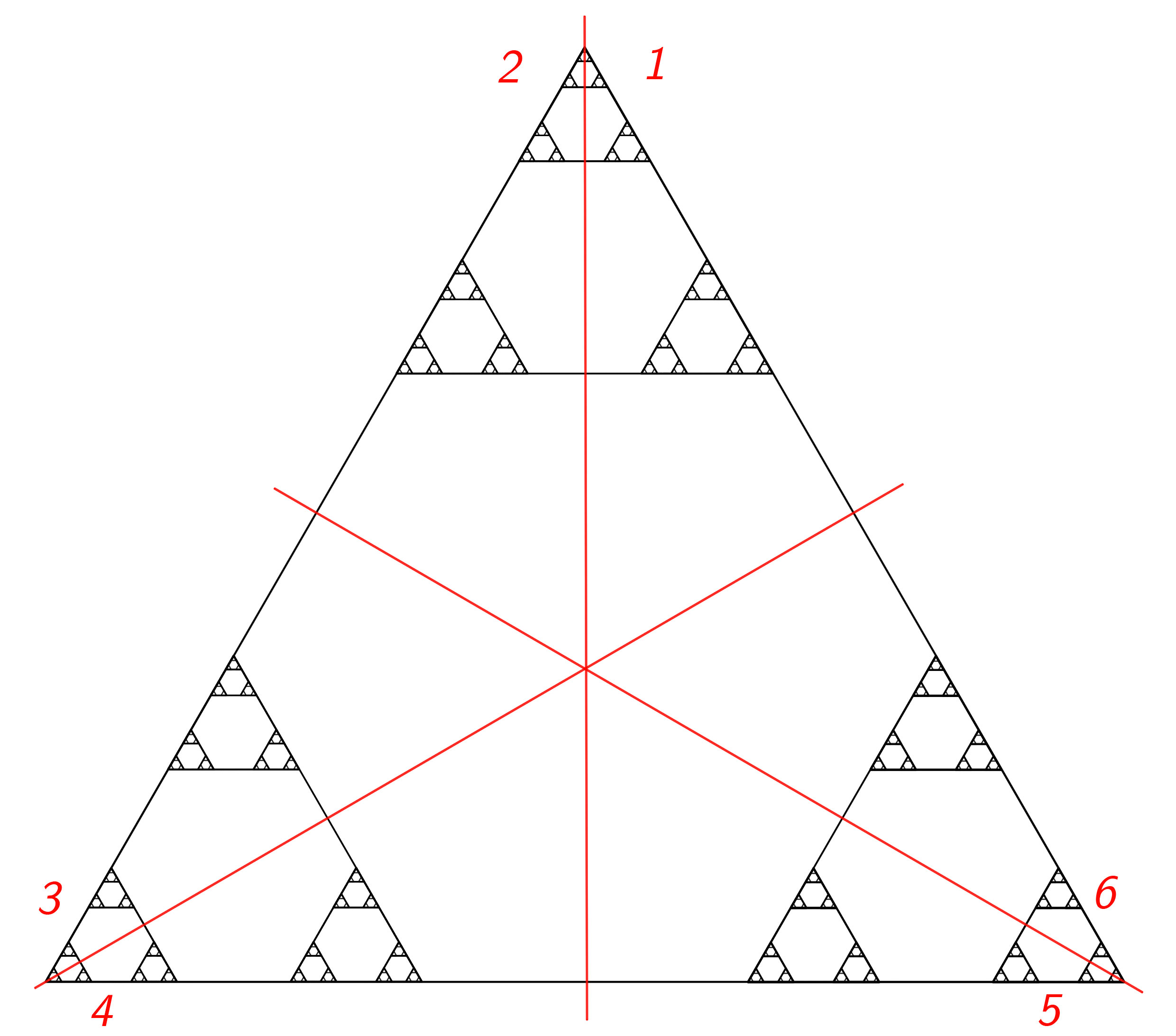}
\caption{Dividing the SSG into six parts.}\label{ssg6}
\end{figure}
The individual parts are denoted by $\tilde K_i$ and the intersections $\tilde V_i:=\tilde K_i \cap (\tilde K_{i+1}\cup\tilde K_{i-1})$ where $i$ is taken modulo $6$. That means $\tilde V_i$ is the intersection of $\tilde K_i$ with the bisecting lines indicated in Figure~\ref{ssg6}. On $\tilde V_i$ we introduce Dirichlet boundary conditions. There are countably infinitely many of those points and all but one lie in the middle of one-dimensional lines.
\begin{align*}
\D_{\R,i}^0:=\{u\in \D_\R : \ u|_{\tilde V_i}\equiv 0,\ \operatorname{supp}(u)\subset \tilde K_i\}
\end{align*}
Denote the parts of the quadratic forms by $\E_{\R,i}:=\E_\R|_{\D_{\R,i}^0\times \D_{\R,i}^0}$.
\begin{lemma}
$(\E_{\R,i},\D_{\R,i}^0)$ is a regular Dirichlet form on $L^2(K,\mu_\eta|_{\tilde K_i\backslash \tilde V_i})$ with discrete non negative spectrum and for $u,v\in \D_{\R,1}^0 \oplus \cdots \oplus \D_{\R,6}^0$ we have
\begin{align*}
\E_\R(u,v)=\sum_{i=1}^6 \E_{\R,i}(u|_{\tilde K_i},v|_{\tilde K_i})
\end{align*}\label{lem6parts}
\end{lemma}
\begin{proof} 
From \cite[Theorem 10.3]{kig12} \cite[Theorem 4.4.3]{fot} we know that $(\E_{\R,i},\D_{\R,i}^0)$ is a regular Dirichlet form on $L^2(K,\mu_\eta|_{\tilde K_i\backslash \tilde V_i})$. The spectrum is discrete since $\D_{\R,i}^0\subset \D_\R$ (see \cite[Theo. 4, Chap. 10]{bs87}). The $\D_{\R,i}^0$ are orthogonal to each other with respect to $\E_\R$ as well as the inner product of $L^2(K,\mu_\eta)$. Therefore we have the desired equality.
 \end{proof}
Let $\varphi$ be any eigenfunction of $(\E_{\R,1},\D_{\R,1}^0)$ with measure $\mu_\eta$ and $\eta\in (0,1]$ with eigenvalue $\lambda$. We can use this $\varphi$ to construct a Dirichlet-Neumann eigenfunction on SSG. By $\tilde \varphi$ we denote the reflection of $\varphi$ along the bisecting line trough $p_1$. I.e. $\tilde \varphi=\varphi\circ \tau$. We glue these functions together in the following fashion which you can see in Figure~\ref{gluing}.
\begin{align*}
\varphi_1&:=\varphi\\
\varphi_2&:=-\tilde \varphi\\
\varphi_3&:=\varphi\circ \sigma^2\\
\varphi_4&:=-\tilde\varphi\circ \sigma^2\\
\varphi_5&:=\varphi\circ \sigma\\
\varphi_6&:=-\tilde\varphi\circ \sigma\\
\end{align*}
\begin{figure}[H]
\centering
\includegraphics[width=0.6\textwidth]{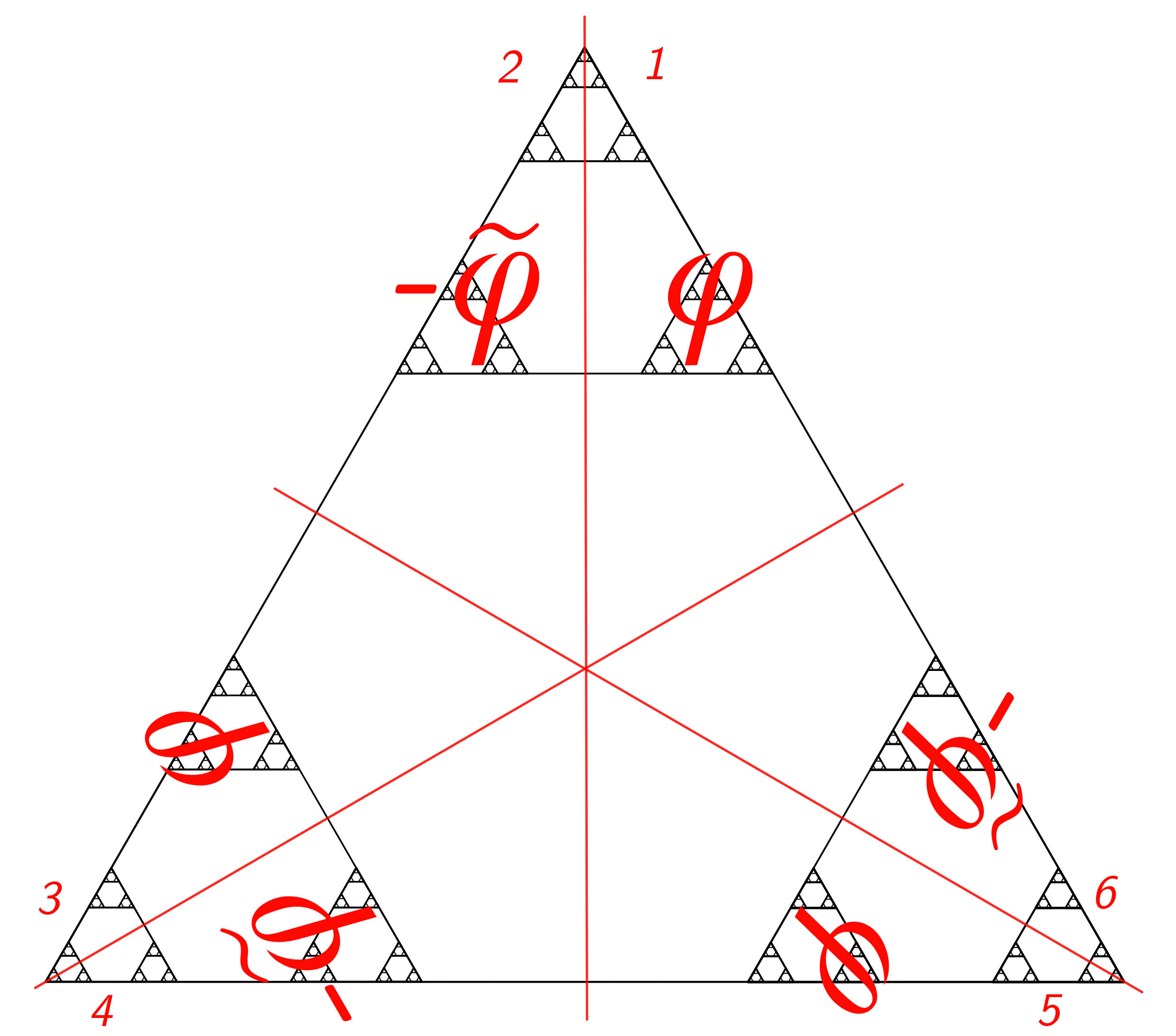}
\caption{Gluing $\protect\varphi$.}\label{gluing}
\end{figure}
We denote the function on $K$ as $\Phi:=\sum_{i=1}^6 \varphi_i$. Thanks to symmetry $\varphi_i$ is an eigenfunction of $(\E_{\R,i},\D_{\R,i}^0)$ with measure $\mu_\eta|_{\tilde K_i}$ and eigenvalue $\lambda$. And the Dirichlet conditions and Lemma~\ref{lem6parts} ensure that $\Phi\in \D_{\R}^0$. We now want to show, that $\Phi$ itself is a Dirichlet-Neumann eigenfunction i.e. $\E_\R(\Phi,v)=\lambda(\Phi,v)_{\mu_\eta}$ for all $v\in \D_\R$. Due to the symmetry we have the following equations:
\setcounter{equation}{0}
\begin{align}
\E_\R(\Phi,v)&=\E_\R(\Phi\circ\tau,v\circ\tau)=\E_\R(\Phi,-v\circ \tau)\label{eq1}\\
\E_\R(\Phi,v)&=\E_\R(\Phi\circ\sigma,v\circ\sigma)=\E_\R(\Phi,v\circ\sigma)\label{eq2}\\
\E_\R(\Phi,v)&=\E_\R(\Phi\circ\sigma^2,v\circ\sigma^2)=\E_\R(\Phi,v\circ\sigma^2)\label{eq3}
\end{align}
From (\ref{eq1}) we get
\begin{align*}
\E_\R(\Phi,v)=\E_\R(\Phi,\underbrace{\tfrac{v-v\circ\tau}2}_{\omega:=})
\end{align*}
Now $\omega$ is anti-symmetric w.r.t. $\tau$ and therefore vanishes on the bisecting line through $p_1$. If we apply (\ref{eq2}) and (\ref{eq3}) to $\omega$ we get
\begin{align*}
\E_\R(\Phi,v)=\E_\R(\Phi,\omega)=\E_\R(\Phi,\underbrace{\tfrac{\omega+\omega\sigma+\omega\sigma^2}3}_{f:=})
\end{align*}
Since $\omega$ vanishes on the bisecting line through $p_1$ we know that $f$ vanishes on $\bigcup_{i=1}^6 \tilde V_i$. That means $f\in \bigoplus_{i=1}^6 \D_{\R,i}^0$ and thus
\begin{align*}
\E_\R(\Phi,v)&=\E_\R(\Phi,f)\\
&\stackrel{(i)}{=}\sum_{i=1}^6\E_{\R,i}(\Phi|_{\tilde K_i},f|_{\tilde K_i})\\
&\stackrel{(ii)}{=}\sum_{i=1}^6 \lambda (\Phi|_{\tilde K_i},f|_{\tilde K_i})_{\mu_\eta|_{\tilde K_i}}\\
&\stackrel{(iii)}{=}\lambda (\Phi,f)_{\mu_\eta}\\
&\stackrel{(iv)}{=}\lambda (\Phi,v)_{\mu_\eta}
\end{align*}
(i) holds due to Lemma~\ref{lem6parts}, (ii) since the parts of $\Phi$ are eigenfunctions, (iii) is clear and (iv) is true since $\mu_\eta$ fulfils the same symmetries as $\E_\R$. Therefore $\Phi$ is a pre-localized eigenfunction.
\end{proof}

We can use the same idea as in the case of the Sierpinski Gasket to get localized eigenfunctions from the pre-localized eigenfunctions. Recall that $\mu_\eta^{(n)}=\mu_\eta^w$ for $|w|=n$. In Lemma~\ref{lemmeas} we showed that this was again a convex combination. Therefore Lemma~\ref{lemdn} is applicable. Let $u^{(n)}$ be a pre-localized eigenfunction with eigenvalue $\lambda_n$ of $(\mathcal{E}_{\mathcal{R}^{(n)}},\mathcal{D}_{\mathcal{R}^{(n)}},\mu_\eta^{(n)})$. Now define for $w\in\mathcal{A}^n$:
\begin{align*}
u_w(x):=\begin{cases}
u\circ G_w^{-1}(x) \ &,x \in K_w\\0 &, \text{elsewhise} 
\end{cases}
\end{align*}
Then $u_w^{(n)}$ is a localized eigenfunction of $(\mathcal{E}_\mathcal{R},\mathcal{D}_\R)$ and $(\mathcal{E}_\mathcal{R},\mathcal{D}_\R^0)$ with measure $\mu_\eta$ and eigenvalue $\frac{\lambda_n}{\delta_n\mu(K_w)}$ in the sense that $\operatorname{supp}(u_w)\subset K_w$.\\

To show this we notice that $u_w^{(n)}\in\mathcal{D}_\R^0$. Since $u^{(n)}\in\mathcal{D}_\R^0$ we have $u_w^{(n)}|_{V_0}\equiv 0$ and $u_w^{(n)}\in C(K)$ and for the finiteness of the quadratic form $\mathcal{E}_\mathcal{R}$:
\begin{align*}
\mathcal{E}_{\mathcal{R}}(u_w^{(n)},u_w^{(n)})&=\frac 1{\delta_n} \mathcal{E}_{\mathcal{R}^{(n)}}(u^{(n)}\circ G_w^{-1}\circ G_w,u^{(n)}\circ G_w^{-1}\circ G_w)\\
&=\frac 1{\delta_n} \mathcal{E}_{\mathcal{R}^{(n)}}(u^{(n)},u^{(n)})\\
&<\infty ,\qquad \text{since } u^{(n)}\in\mathcal{D}_{\mathcal{R}^{(n)}}
\end{align*}

Now for all $v\in\mathcal{D}_\R$:
\begin{align*}
\mathcal{E}_\mathcal{R}(u_w^{(n)},v)&=\frac 1{\delta_n}\mathcal{E}_{\mathcal{R}^{(n)}}(u^{(n)},\underbrace{v\circ G_w}_{\in\mathcal{D}_{\mathcal{R}^{(n)}}})\\
&=\frac 1{\delta_n}\lambda_n(u^{(n)},v\circ G_w)_{\mu_\eta^w}\\
&=\frac {\lambda_n}{\delta_n}\int_K u^{(n)} \cdot v\circ G_w \ d\mu_\eta^w\\
&=\frac {\lambda_n}{\delta_n}\frac 1{\mu(K_w)}\int_{K_w} u^{(n)}\circ G_w^{-1} \cdot v\  d\mu_\eta \\
&=\frac {\lambda_n}{\delta_n\mu_\eta(K_w)} (u_w^{(n)},v)_{\mu_\eta}
\end{align*}
Therefore $u_w^{(n)}$ is an eigenfunction of $(\E_\R,\D_\R, \mu_\eta)$. We can do this for any $n$-cell and therefore the multiplicity of the eigenvalue $\frac{\lambda_n}{\delta_n\mu_\eta(K_w)}$ is at least $3^n$ since there are that many $n$-cells. We define $\mu_\eta(n):=\mu_\eta(K_w)$.\\

This means we have shown the following result.
\begin{lemma}
Let $\R$ be any sequence of matching pairs and $\mu_\eta$ with $\eta\in (0,1]$. Then for all $n\in \mathbb{N}$ and $w\in \A^n$ there exists an eigenfunction $u_w$ of $-\Delta_D^{\mu_\eta,\R}$ with $\operatorname{supp}(u_w)\subset K_w$ and multiplicity of the corresponding eigenvalue at least $3^n$.
\end{lemma}

However the pre-localized eigenfunction $u^{(n)}$ with eigenvalue $\lambda_n$ depend on $\mathcal{R}^{(n)}$ and $\mu_\eta^{(n)}$ and thus it is a different one for every $n$. This means $\lambda_n$ may be not the same for all $n$. This is a different situation in comparison to the self-similar case. In the case of the Sierpinski Gasket there was only one pre-localized eigenfunction necessary.

The other scaling parameters $\delta_n$ and $\mu_\eta(n)$ are the right ones but to be able to calculate the growing rate of the eigenvalues of localized eigenfunctions we need further information about $\lambda_n$ such as if it is bounded. Since our proof of Lemma~\ref{lemdn} was slightly different than the one in \cite{bk97}	we can use this to get estimates on $\lambda_n$.\\

In the proof of Lemma~\ref{lemdn} we saw that the eigenvalue of $\Phi$ is the same as the one of $(\E_{\R,1},\D_{\R,1}^0)$ with $\mu_\eta$. We are able to get estimates on the first eigenvalue there by
\begin{align*}
\lambda_1=\inf_{u\in \D_{\R,1}^0}\frac{\E_\R(u,u)}{||u||_{\mu_\eta}^2}
\end{align*}
We want to find a function $u\in \D_{\R,1}^0$ such that this value is bounded uniformly from above for all $\mathcal{R}^{(n)}$ and all $\mu_\eta$ with $\eta\in(0,1]$. To do this we choose the values of $u$ on $V_3\cap \tilde K_1$ as indicated in Figure~\ref{ssghalfdirich} and extend harmonically. We only consider sequences $\R$ that fulfil the conditions of chapter~\ref{cond}.
\begin{figure}[H]
\centering
\includegraphics[width=0.6\textwidth]{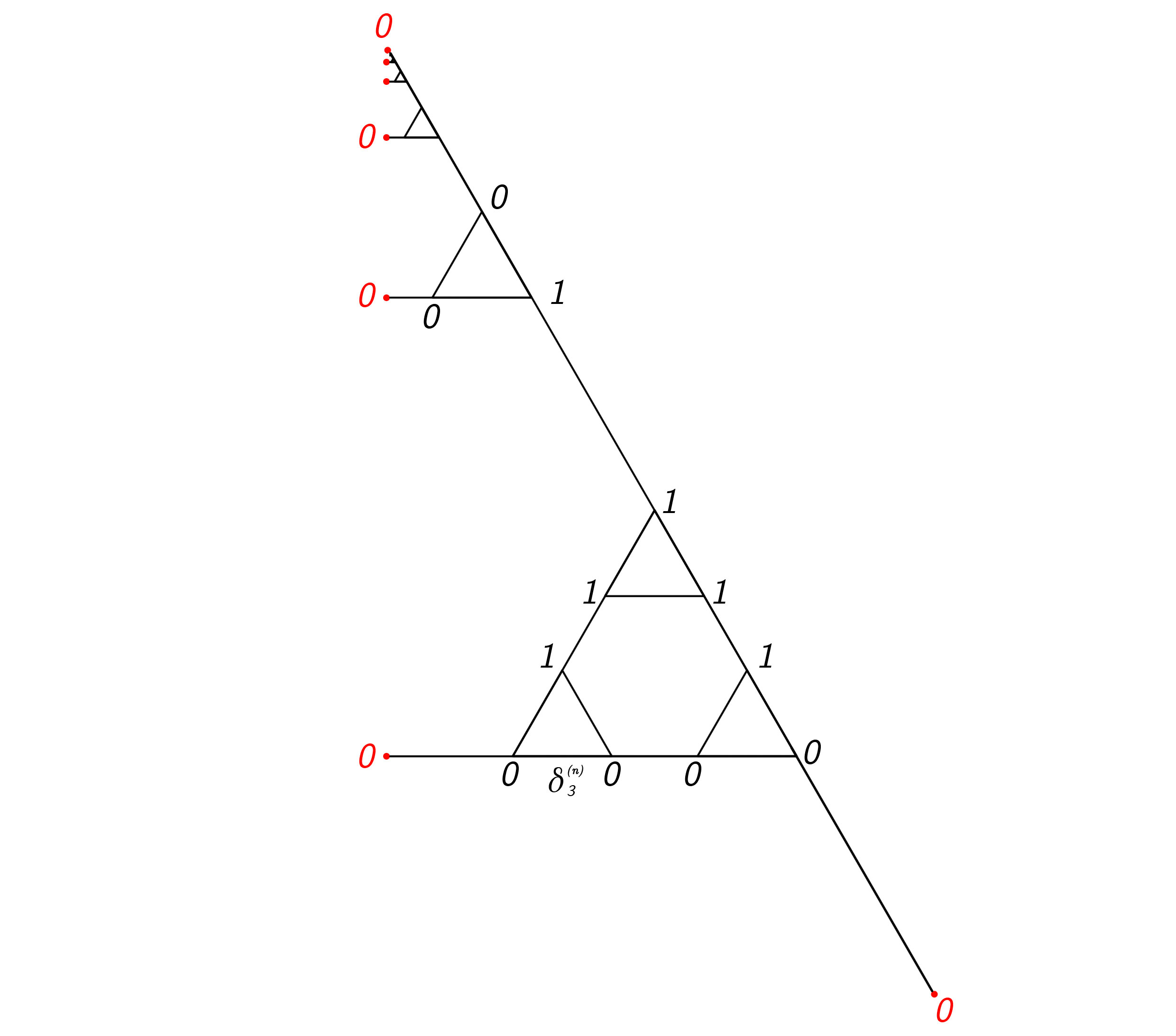}
\caption{Construction of $\protect u$ on $\protect\tilde K_1$.}\label{ssghalfdirich}
\end{figure}
\noindent Then the energy of $u$ is
\begin{align*}
\E_{\R^{(n)}}(u,u)=6\frac 1{\delta_3^{(n)}} \leq \frac 6{\kappa_1 r^3}
\end{align*}
Also since we have a $3$-cell $K_{\tilde w}$ where $u$ is constant $1$ we get an estimate on the $L^2$ norm.
\begin{align*}
||u||_{\mu_\eta}^2=\int_{\tilde K_1} u^2 d\mu_\eta\geq \int_{K_{\tilde w}} 1 d\mu_\eta=\mu_\eta(K_{\tilde w})=\eta \beta^3+(1-\eta)(\tfrac 13)^3\geq \beta^3 
\end{align*}
Therefore the first eigenvalue of $(\E_{\R^{(n)},1}, \D_{\R^{(n)},1}^0)$ with measure $\mu_\eta|_{\tilde K_1\backslash \tilde V_1}$ is bounded by 
\begin{align*}
\lambda_1\leq \frac 6{\kappa_1 r^3\beta^3}
\end{align*}
This constant is independent of $n$ as well as $\eta$.\\

Therefore the localized eigenfunctions on SSG give us a sequence of eigenvalues 
\begin{align*}
\nu_n=\frac{\lambda_n}{\delta_n\mu_\eta(n)}\leq \begin{cases} \tilde c \left(\frac 3r\right)^n&, \eta\in(0,1)\\ 
\tilde c (\beta r)^{-n}&, \eta=1
\end{cases}
\end{align*}
with multiplicities at least $3^n$. \\

For completeness we also want to get a lower bound for $\lambda_n$. We do not need really need it for the argument, nonetheless it shows that the localized eigenfunctions are indeed responsible for the asymptotic growing of the eigenvalue counting function.

$\lambda_n$ is the first eigenvalue of $(\E_{\R^{(n)}},\D_{\R^{(n)}},\mu^{(n)}_\eta)$. From \cite[Lemma 7.19]{afk17} we know that
\begin{align*}
|u(p)-u(q)|^2\leq 16 \E_\R(u,u), \ \forall u\in D_\R
\end{align*}
Essentially this means, that the diameter of $K$ with respect to the resistance metric is bounded by $16$, which is in particular independent of $\R$. For $u\in \D_{\R,1}^0$ and $p_1\in V_0$ we have
\begin{align*}
\E_\R(u,u)&\geq \frac 1{16} |u(x)-u(p_1)|^2\\
&= \frac 1{16} |u(x)|^2\\[.2cm]
\Rightarrow \E_\R(u,u)&\geq \frac 1{16}\int_K |u(x)|^2d\mu_\eta\\
&=||u||^2_{\mu_\eta} 
\end{align*}
This means the first Dirichlet eigenvalue of $(\E_1,\D_{\R,1}^0)$ with measure $\mu_\eta$ is at least $\frac 1{16}$. This is independent of the measure $\mu_\eta$ and the sequence of matching pairs $\R$. Thus we have
\begin{align*}
\lambda_n\geq \frac 1{16}
\end{align*}
All together we have found a sequence of eigenvalues $\nu_n$ with multiplicities at least $3^n$ such that
\begin{align*}
\left.
\begin{array}{ll}
\tilde c_1 \left(\frac 3r\right)^n\\
\tilde c_1 \left( \beta r\right)^{-n}
\end{array}
\right\}
\leq \nu_n\leq\begin{cases}
 \tilde c_2 \left(\frac 3r\right)^n&, \eta\in(0,1)\\
 \tilde c_2 \left( \beta r\right)^{-n}&, \eta=1
\end{cases}
\end{align*}
with constants $\tilde c_1, \tilde c_2$ independent of $n$. \\

We are now able to show that we can't have convergence. Again we want to emphasize that the upper estimate is the one we need. The lower estimate was actually already implied by Theorem~\ref{theo24}.\\

Consider $\eta\in(0,1)$. From Theorem~\ref{theo24} we know that 
\begin{align*}\limsup_{x\rightarrow\infty} N_D^{\mu_\eta,\R}(x)x^{-\frac 12 d_S}-\liminf_{x\rightarrow\infty} N_D^{\mu_\eta,\R}(x)x^{-\frac 12 d_S}<\infty\end{align*}

With $N_D^{\mu_\eta,\R}(x)_-:=\lim\limits_{\epsilon \searrow 0}N_D^{\mu_\eta,\R}(x-\epsilon)$ and $\mu_\eta(n):=\mu_\eta(K_w)$ for $|w|=n$:
\begin{align*}
\limsup_{x\rightarrow\infty} &N_D^{\mu_\eta,\R}(x)x^{-\frac 12 d_S}-\liminf_{x\rightarrow\infty} N_D^{\mu_\eta,\R}(x)x^{-\frac 12 d_S}\\
\geq &\limsup_{n\rightarrow\infty} N_D^{\mu_\eta,\R}\left(\frac{\lambda_n}{\delta_n\mu_\eta(n)}\right)\cdot \left(\frac{\lambda_n}{\delta_n\mu_\eta(n)}\right)^{-\frac 12 d_S}\\
&\pushright{-\liminf_{n\rightarrow\infty} N_D^{\mu_\eta,\R}\left(\frac{\lambda_n}{\delta_n\mu_\eta(n)}\right)_-\cdot \left(\frac{\lambda_n}{\delta_n\mu_\eta(n)}\right)^{-\frac 12 d_S}}\\
\geq& \lim_{n\rightarrow \infty}\left(N_D^{\mu_\eta,\R}\left(\frac{\lambda_n}{\delta_n\mu_\eta(n)}\right)-N_D^{\mu_\eta,\R}\left(\frac{\lambda_n}{\delta_n\mu_\eta(n)}\right)_-\right)\cdot \left(\frac{\lambda_n}{\delta_n\mu_\eta(n)}\right)^{-\frac 12 d_S}\\
\geq &\lim_{n\rightarrow\infty} 3^n\cdot\left(\frac{\lambda_n}{\delta_n\mu_\eta(n)}\right)^{-\frac 12 d_S}\\
\geq &\lim_{n\rightarrow\infty} 3^n\cdot\left(\tilde c_2 (\tfrac 3r)^n\right)^{-\frac 12 d_S}\\
= &\lim_{n\rightarrow \infty} 3^n \cdot \tilde c_2^{-\frac 12 d_S}\cdot 3^{-n}\\
= & \tilde c_2^{-\frac 12 d_S}>0
\end{align*}
The calculation for $\eta=1$ is analogous. This closes the proof of the main theorem of this work:
\begin{theorem}[Non-convergence on SSG]\ \\
Let $\R$ be a sequence of matching pairs that fulfils the conditions and $\mu=\mu_\eta$ with $\eta\in(0,1]$. Let $\beta>\frac 1{9r}$ if $\eta=1$. Then we have
\begin{align*}
\liminf_{x\rightarrow \infty} N_D^{\mu,\R}(x)x^{-\frac12 {d_S}} < \limsup_{x\rightarrow \infty} N_D^{\mu,\R}(x)x^{-\frac12 {d_S}}
\end{align*}
\end{theorem}

\section{Special cases with more symmetry}\label{special}
We were not able to show periodicity in the general setting. This is due to the very general setting of sequences of matching pairs. We have a lot of options to choose this sequence. This destroys the symmetry that we need to show periodicity. But we can look at some special sequences to give us back the symmetry. \\

One problem was, that the measures $\mu_\eta$ were not self-similar for $\eta\in (0,1)$. We need to get estimates on the $L^2$ norms. To avoid this we could choose the measure $\mu_1=\mu_l$ which by itself is self-similar in the sense that $\mu_l^{(n)}=\mu_l$ for all $n$. \\

With the estimates between $\E_\R$ and $\E_{\R^{(n)}}$ we reached the following estimates for the eigenvalue counting function in chapter~\ref{chapper}.
\begin{align*}
3^nN_D(c_1(\beta r)^nx)\leq N_D(x)\leq 3^n N_D(c_2(\beta r)^nx) + 3^{n+1} + N(\mathcal{E},D_{J_n},\mu,x)
\end{align*}
with $c_1<1<c_2$ for all $n$.\\

In general these constants $c_1$ and $c_2$ are not $1$. Even not asymptotically! However there is a case where they are $1$. This is the case if the quadratic forms $\E_\R$ and $\E_{\R^{(n)}}$ coincide. This can be achieved by choosing constant sequences of matching pairs with $r\in [\frac 13, \frac 35)$.
\begin{align*}
(r_i,\rho_i)=(r,\rho)  \ \forall i
\end{align*}
If the sequence is constant we have $\mathcal{R}=\mathcal{R}^{(n)} \ \forall n$ and thus
\begin{align*}
\mathcal{E}_{\mathcal{R}}=\mathcal{E}_{\mathcal{R}^{(n)}}
\end{align*}
This means we get the following rescaling of the eigenvalue counting function.
\begin{align*}
3N_D(\beta rx )\leq N_D(x)\leq 3N_D(\beta rx)+9+N(\mathcal{E}_\R,\mathcal{D}_{\R,J_1},\mu_l,x)
\end{align*}
We want to apply renewal theory to the Dirichlet eigenvalue counting function to show the existence of log-periodic behaviour. This version of the renewal theorem we use can be found in \cite{kig98}. This is a refinement of the version from \cite{kl93}. The original version is due to Feller \cite{fe66}.\\

\begin{theorem}[Renewal theorem, Kigami {\cite[Theorem A.1]{kig98}}]\label{renewal}\	 \\
Let $f$ be a measurable function on $\mathbb{R}$ with $f(t)\rightarrow 0$ as $t\rightarrow -\infty$. Suppose $f$ satisfies a renewal equation
\begin{align*} f(t)=\sum_{j=1}^Nf(t-m_jT)p_j+u(t)\end{align*}
where $m_1,m_2,\ldots,m_N$ are positive integers whose greatest common divider is $1$, $\sum_{j=1}^Np_j=1$ and $p_j>0$ for all $j$. Also assume that $\sum_{j=-\infty}^\infty |u_j(t)|$ converges uniformly on $[0,T]$, where $u_j(t)=u(t+jT)$ for $t\in[0,T]$. Set $f_n(t)=f(t+nT)$ for $n\in \mathbb{Z}$ and $G(t)=(\sum_{j=1}^N m_jp_j)^{-1} \sum_{j=-\infty}^\infty u_j(t)$. Then as $n\rightarrow \infty$, $f_n$ converges to $G$ uniformly on $[0,T]$. 

Moreover, set $Q(z)=(1-\sum_{j=1}^Np_jz^{m_j})/(1-z)$ and define $\beta=\min\{|z|:Q(z)=0\}$ and $m=\max\{multiplicity \ of\  Q(z)=0\  at\  q: |w|=\beta, Q(w)=0\}$. If there exist $C>0$ and $\alpha>1$ such that $|u(t)|\leq C\alpha^{-t}$ for all $t$, then, as $t\rightarrow \infty$, 
\begin{align*}
|G(t)-f(t)|=\begin{cases} \mathcal{O}(t^{m-1}\beta^{-t/T}) & if \ \alpha^T>\beta,\\
\mathcal{O}(t^m\alpha^{-t}) & if \ \alpha^T=\beta,\\
\mathcal{O}(\alpha^{-t}) & if \ \alpha^T<\beta.
\end{cases}
\end{align*}
\end{theorem}
We define
\begin{align*}
R(x)&:=N_D(x)-3N_D\left( \beta r x\right)\\
f(t)&:=e^{-td_S}N_D(e^{2t})\\
u(t)&:=e^{-td_S}R(e^{2t})\\
T&:=-\ln\sqrt{(\beta r)}
\end{align*}
We see that $f$ is measurable and $f(t)\rightarrow 0$ for $t\rightarrow -\infty$. With $N=3$, $m_j=1$ and $p_j=\frac 13$ for all $j$ we have
\begin{align*}
\sum_{j=1}^N f(t+\ln\sqrt{\beta r})\frac 13&=f(t+\ln\sqrt{\beta r})\\
&=e^{-(t+\ln\sqrt{\beta r})d_S}N_D(e^{2(t+\ln\sqrt{\beta r})})\\
&=e^{-td_S}\left(\beta r\right)^{-\frac {d_S}2}N_D(e^{2t}\beta r)\\
&=e^{-td_S}3N_D(e^{2t}\beta r)\\
&=e^{-td_S}(N_D(e^{2t})-R(e^{2t}))\\
&=f(t)-u(t)
\end{align*}
Therefore they fulfil the renewal equation.\\

We need to show that $\sum_{j=-\infty}^\infty |u(t+jT)|$ converges uniformly for all $t\in[0,T]$. The first Dirichlet-eigenvalue $\lambda_1^D$ is positive and thus $R(x)=0$ for all $x< \beta r \lambda_1^D$. This means there is a $j_0$ such that 
\begin{align*}
\sum_{j=-\infty}^\infty |u(t+jT)|=\sum_{j=j_0}^\infty |u(t+jT)| 
\end{align*}
Since $R(x)=\mathcal{O}(x^{\frac 12})$ we get a constant $c\geq 0$ such that \setcounter{equation}{0}
\begin{align}
u(t)&=e^{-d_S t}R(e^{2t})\leq c e^{-t(d_S-1)} \label{eqorder}
\end{align}
and thus
\begin{align*}
\sum_{j=-\infty}^\infty |u(t+jT)|&=\sum_{j=j_0}^\infty |u(t+jT)| \\
&\leq \sum_{j=j_0}^\infty c e^{-(t+jT)(d_S-1)}\\
&=ce^{-t(d_S-1)}\sum_{j=j_0}^\infty e^{-jT(d_S-1)}
\end{align*}
Since $d_S>1$ the sum converges independent of $t$ and thus we have uniform convergence. We have $Q(z)=1$ and thus $\beta=\infty$ which means we are in the third case of the theorem. With (\ref{eqorder}) we get
\begin{align*}
|G(t)-f(t)|\leq \tilde c e^{-t(d_S-1)}
\end{align*}
Substituting $x=e^{2t}$:
\begin{align*}
|G(\tfrac{\ln x}2)-x^{-\frac 12 d_S}N_D(x)|\leq \tilde cx^{-\frac 12 (d_S-1)}
\end{align*}
This means, that 
\begin{align*}
N_D(x)=G\left(\frac{\ln x}2\right)x^{\frac 12 d_S}+\mathcal{O}(x^{\frac 12})
\end{align*}
The function $G$ is periodic with period $T=-\ln\sqrt{\beta r}$.
With the results from chapter~\ref{chaple} we know that $G$ is nonconstant. \\[.2cm]

Another special case is if $\mathcal{R}$ is periodic. If it is not just constant we are not in the case that $r_n\rightarrow r$. But still we get the asymptotics. Let $\mathcal{R}$ be a sequence of matching pairs, such that 
\begin{align*}
\exists n\in\mathbb{N}: \ \mathcal{R}^{(n)}=\mathcal{R}
\end{align*}
With $\mu=\mu_l$ we get the following estimates for the eigenvalue counting functions:
\begin{align*}
3^{n} N_D(\beta^{n} \delta_n x)\leq N_D(x)\leq 3^{n} N_D(\beta^{n} \delta_n x)+ 3^{n+1}+ N(\mathcal{E},\mathcal{D}_{J_{n}},\mu_\eta,x)
\end{align*}
If $\delta_n^{\frac 1n}\in [\tfrac 13, \tfrac 35)$ and $\beta>(\delta_n^{\frac 1n}9)^{-1}$ we get the asymptotic growing  
\begin{align*}
d_S=\frac{\ln 9}{-\ln(\beta \delta_n^{\frac 1n})}
\end{align*}
In this case we can also apply the renewal theorem in Theorem~\ref{renewal} and the existence of localized eigenfunctions from chapter~\ref{chaple} and get that $N_D(x)x^{-\frac 12 d_S}$ does not converge. In particular there exists a non constant periodic function $G$ with period $-\tfrac 12\ln(\beta \delta_n^{\frac 1n})$ such that
\begin{align*}
N_D(x)=G\left(\frac{\ln x} 2\right)x^{\frac 12 d_S}+ \mathcal{O}(x^{\frac 12})
\end{align*}
It is easy to see, that this is the only case with $\mathcal{E}_{\mathcal{R}^{(n)}}=\mathcal{E}_\mathcal{R}$. That means, in all other cases the constants $c_1,c_2$ are not equal.

\section*{Acknowledgements}
Most of this work originated during a stay at Kyoto University. I want to thank Professor Jun Kigami for his hospitality and his support. Many of the ideas developed during numerous fruitful discussions with Professors Naotaka Kajino and Jun Kigami for which I am very grateful. This stay was supported by the  German Academic Exchange Service (DAAD) with a "DAAD-Doktorandenstipendium". I am very thankful for the assistance, since without it, this stay and thus this work would not have been possible.

\end{document}